\documentclass[11pt]{amsart}
\usepackage{amscd, amssymb}
\usepackage{graphics}

\theoremstyle{plain}

\newtheorem{theorem}{Theorem}
\newtheorem*{question}{Question}
\newtheorem*{questions}{Questions}

\numberwithin{equation}{section}

\newtheorem{proposition}[equation]{Proposition}

\newtheorem{lemma}[equation]{Lemma}
\newtheorem{corollary}[equation]{Corollary}
\theoremstyle{definition}
\newtheorem{remark}[equation]{Remark}
\newtheorem{example}[equation]{Example}

\newcommand {\printname}[1] {}

\def    \R  {{\mathbb R}}
\def    \Z  {{\mathbb Z}}

\def    \F  {{\mathbb F}}

\def    \CP {{\mathbb {CP}}}
\def    \P {{\mathbb {P}}}
\def    \C  {{\mathbb C}}
\def    \N  {{\mathbb N}}
\def    \deg  {{\operatorname{deg}}}
\def    \Hom  {{\operatorname{Hom}}}
\def    \Ext  {{\operatorname{Ext}}}
\def    \even  {{\operatorname{even}}}
\def    \rank  {{\operatorname{rank}}}
\def    \torsion  {{\operatorname{torsion}}}
\def    \lot  {{\operatorname{\ lower\ order\ terms }}}

\def    \Tilde  {\widetilde}
\def    \Hat  {\widehat}

\def    \Mh     {\Hat{M}}
\def    \oh     {\Hat{\omega}}
\def    \Fh     {\Hat{F}}
\def    \ph     {\Hat{\phi}}
\def    \at     {\Tilde{\alpha}}
\def    \bt     {\Tilde{\beta}}
\def    \mt     {\Tilde{\mu}}
\def    \ut     {\Tilde{u}}
\def    \xt     {\Tilde{x}}
\def    \yt     {\Tilde{y}}

\def    \Gt     {\Tilde{G}}

\begin{document}
\title{Hamiltonian circle actions with minimal fixed sets}

\author{Hui Li and Susan Tolman}
\address{School of Mathematical Sciences, Box 173\\
        Suzhou University\\
        Suzhou, 215006, China.}
        \email{hui.li@suda.edu.cn}

\address{Department of mathematics\\
        University of Illinois at Urbana-Champaign\\
        Urbana, IL, 61801, USA.}
\email{stolman@math.uiuc.edu}

\thanks{2010 classification. Primary$\colon$ 53D05, 53D20; Secondary$\colon$
55N25, 57R20.}
\thanks{The second author was partially supported by National Science Foundation
Grant DMS \#07-07122.}
\begin{abstract}
Consider an  effective  Hamiltonian circle action
on a compact symplectic $2n$-dimensional manifold $(M, \omega)$.
Assume that the fixed set $M^{S^1}$ is {\em minimal},
in two senses: it has exactly two components,
$X$ and $Y$, and  $\dim(X) + \dim(Y) = \dim(M) - 2$.

We prove that the integral
cohomology ring and Chern classes of $M$ are isomorphic to either
those of $\CP^n$ or (if $n \neq 1$ is odd) to those of $\Gt_2(\R^{n+2})$,  the Grassmannian of oriented
two-planes in $\R^{n+2}$.
In particular,
$H^i(M;\Z) = H^i(\CP^n;\Z)$ for all $i$, and
the Chern classes of $M$ are determined by the integral
cohomology {\em ring}.
We also prove that the fixed set data
of $M$ agrees exactly with the fixed set data for one of the standard circle actions on
one of these two manifolds.
In particular, we show that there are no points
with stabilizer $\Z_k$ for any $k > 2$.

The same  conclusions hold when $M^{S^1}$ has exactly two components
and the even Betti numbers of $M$ are minimal, that is,
$b_{2i}(M) = 1$ for all $i \in  \left\{0,\dots,\frac{1}{2}\dim(M)\right\}$.
This provides additional evidence
that very few symplectic manifolds with minimal even Betti numbers
admit Hamiltonian actions.

\end{abstract}
 \maketitle

\section{Introduction}

Let $M$ be a connected compact smooth manifold, and let $G$ be a
connected compact Lie group.
One fundamental question is whether $M$ admits a $G$ action, and -- if
so -- how many different actions can occur.

In the early 1970's, Petrie addressed this  question by studying
{\em homotopy complex projective spaces}, spaces which are homotopy
equivalent to $\CP^n$. He conjectured  that if a homotopy complex
projective space $M$ admits a circle action\footnote{ Unless we
specify otherwise, we shall always assume that our actions are
non-trivial.}, then the total Pontryagin class $p(M)$ of $M$ agrees
with $p(\CP^n)$ \cite{P1}.  If true, this conjecture would imply
that up to diffeomorphism, there are only finitely many homotopy
complex projective spaces which admit a circle action  for all $n
\neq 2$. Although Petrie's conjecture is still open, it has been
proved in many  special interesting cases, including when the fixed
point set  $M^{S^1}$ contains at most four connected components
\cite{W, Yo76, TsWa, Ma1}. In contrast, a circle action on a complex projective space may not be
equivariantly diffeomorphic to one of the standard actions on $\CP^n$;
Petrie was able to construct {\em exotic} circle actions on $\CP^3$, that is, actions whose normal
representations at the fixed points do not agree with those of any
of the actions induced by a circle subgroup $S^1 \subset SU(4)$
\cite{P2}.

We are interested in addressing the analogous questions for
symplectic manifolds. Because extra tools are available in this
case, we hope to be able to relax the restrictions on $M$ and still
draw stronger conclusions. More precisely, let the circle $S^1$ act
 on a compact symplectic manifold $(M,\omega)$ with moment map
$\phi \colon M \to \R$. Since $M$ is compact and symplectic, $
H^{2i}(M;\R) \neq 0$ for all $i$ such that $0 \leq 2i \leq \dim(M).$
 Since $\phi$ is a
Morse-Bott function and its critical set is the fixed set $M^{S^1}$,
this implies that
\begin{equation*}
\sum_{F \subset M^{S^1}} \left( \dim(F)  + 2 \right) \geq \dim(M) + 2,
\end{equation*}
where the sum is over all fixed components. We say that the fixed
components have {\bf minimal dimension} if
\begin{equation*}
\sum_{F \subset M^{S^1}} \left( \dim(F)  + 2 \right) =
\dim(M) + 2.
\end{equation*}
A priori, this assumption  does not seem terribly restrictive. For
example, if $H^*(M;\R) = H^*(\CP^n;\R)$, then $M$ has fixed
components of minimal dimension, but the converse is false; see
Lemma~\ref{index} and Remark~\ref{compare}. Nevertheless, we are
interested in answering the following questions:

\begin{questions}\label{q1}
Consider a Hamiltonian circle action on a compact symplectic
manifold $(M,\omega)$ with fixed components of minimal dimension.
\begin{itemize}
\item [(a)] Is $b_{2i+1}(M) = 0$ for all $i$? Is $H^*(M;\Z)$ torsion free?
\item [(b)] Can we list all possible cohomology rings $H^*(M;\Z)$?
\item [(c)] Are the Chern classes\footnote{See
\S\ref{background} for the definition of the Chern classes of a symplectic manifold.}
$c_i(M)$  determined by the
cohomology ring $H^*(M;\Z)$?
 \item [(d)] Do the normal representations at the fixed points always agree with those
of some ``standard" circle actions?
\end{itemize}
\end{questions}

In the case that $M$ is $6$-dimensional, all four questions are
answered affirmatively in \cite{T}; see Remark~\ref{compare}.
Partial results also follow from \cite{Dej, Yo77, Ma3, Ah, AO}.
Additionally, if $M$ is $6$-dimensional and the fixed set is
discrete then $M$ is equivariantly symplectomorphic to some
well-known K\"ahler example with additional symmetries \cite{Mc}.
Similarly, Morton has answered all four questions affirmatively
whenever a torus $T$ acts on $M$ so that $M$ is a GKM space and
$\dim T \geq \frac{1}{4} \dim M$ \cite{Mo}. Finally, under the
above assumptions, let $k$ be the largest integer
that divides $c_1(M) \in H^2(M;\Z)$. If the fixed set is discrete
then \cite{Mu, Ha} imply that  $k \leq n + 1$; moreover, if $k =
n+1$ or if  $k = n$ and $n \neq 1$ is odd, then the normal representations
at the fixed points agree with those of one of the standard actions
on $\CP^n$ or $\Tilde G_2(\R^{n+2})$, the Grassmannian of oriented
two-planes in $\R^{n+2}$, respectively.

In this paper, we answer
the questions above
in the case that the fixed set has
the smallest possible number of components -- two.
More concretely, we assume that $M^{S^1}$ has exactly two components,
$X$ and $Y$, and that $\dim(X) + \dim(Y) = \dim(M) - 2$.
It is easy to see that $\CP^n$ admits a Hamiltonian circle action satisfying these assumptions;
if $n \neq 1$ is odd then $\Tilde G_2(\R^{n+2})$ does as well;
see Examples~\ref{ex1} and \ref{ex2}.

Our first main theorem is that, under the above assumptions,
the integral cohomology ring and
the total Chern class of $M$ are isomorphic to those of one of these
two manifolds. (By Proposition~\ref{C}, the same claims hold for the
{\em equivariant} cohomology ring and total Chern class.)

\begin{theorem}\label{A}
Let the circle  act in a Hamiltonian fashion on a  compact $2n$-dimensional
symplectic manifold $(M, \omega)$.
Assume that  $M^{S^1}$
has exactly two components, $X$ and $Y$,
and that $\dim(X) + \dim(Y) = \dim(M) - 2$.
Then one of the following is true:
\begin{itemize}
\item [(A)]$H^*(M; \Z)=\Z[x]/(x^{n + 1})$ and
$c(M)=(1+x)^{n + 1}$; or
\item [(B)] $n \neq 1$
is odd, $H^*(M;\Z) =
\Z[x,y]/\big(x^{\frac{1}{2}(n+1) } - 2y, y^2\big)$,
and \\
$c(M) = \frac
{(1+x)^{n + 2}}
{1+2x}
 .$
\end{itemize}
In both cases, $x$ has degree $2$; in case (B),
$y$ has degree $n +1$.
\end{theorem}

Our second main theorem is that the fixed set data
of $M$ agrees exactly with the fixed set data for one of the standard circle actions on
one of these two manifolds.
Here, the {\bf fixed set data} is
the integral cohomology ring and the total
Chern class of each fixed component, the set of weights for the
normal representation at each fixed component, and for each  weight
$k$, the total Chern class of the subbundle (of the normal bundle of
the fixed component)  on which $S^1$ acts with weight $k$.

\begin{theorem}\label{B}
Let the circle  act effectively\footnote{
A group $G$ acts {\bf effectively} on $M$ if for every non-trivial
$g \in G$, there exists $m \in M$ so that $g \cdot m \neq m$.}
on a  compact symplectic manifold
$(M, \omega)$  with moment map $\phi \colon M \to \R$.
Assume that $M^{S^1}$  has
exactly two components,  $X$ and $Y$, and
that $\dim(X) + \dim(Y) = \dim(M) - 2.$
Then
\begin{gather*}
H^*(X;\Z) = \Z[u]/u^{i  + 1} \ \  \mbox{and} \ \  c(X) = (1+u)^{i+1},
\quad \mbox{where} \  \dim(X) = 2i; \\
H^*(Y;\Z) = \Z[v]/v^{j + 1} \ \  \mbox{and} \ \  c(Y) = (1+v)^{j+1},
\quad \mbox{where} \  \dim(Y) = 2j.
\end{gather*}
Moreover, one of the following is true:
\begin{itemize}
\item  [(A)] The action is semifree\footnote{
A circle action is {\bf semifree} if
the action is free outside the fixed point set.},
\begin{gather*}
c(N_X) = (1+u)^{j + 1},
\quad \mbox{and} \quad
c(N_Y) = (1+v)^{i + 1},
\end{gather*}
where $N_X$ and $N_Y$ are the normal bundles to $X$ and $Y$, respectively.
\item [(B)] The action is not semifree,
but  no point has stabilizer $\Z_k$ for any  $k > 2$; moreover,
\begin{align*}
\dim(X) = \dim(Y)\geq 2 &\quad \mbox{and} \quad \dim(M^{\Z_2}) = \dim(M) - 2; \\
c\big(N_{M^{\Z_2}}\big)\big|_X = 1+ u &\quad \mbox{and} \quad
c\big(N_{M^{\Z_2}}\big)\big|_Y = 1 + v; \\
c\big(N_X^{M^{\Z_2}}\big) = \frac
{(1+u)^{i + 1}} {1 + 2u }
&\quad \mbox{and} \quad
c\big(N_Y^{M^{\Z_2}}\big) = \frac
{(1+v)^{i + 1}} {1 + 2v},
\end{align*}
where
$M^{\Z_2}$ denotes the submanifold fixed by $\Z_2$,
$N_{M^{\Z_2}}$ denotes the
normal bundle of $M^{\Z_2}$ in $M,$
and $N_X^{M^{\Z_2}}$ and $N_Y^{M^{\Z_2}}$ denote the normal bundles of $X$ and of $Y,$
respectively, in $M^{\Z_2}.$
\end{itemize}
\end{theorem}

\begin{example}\label{ex1}
Given $n \geq 1$, let $\CP^n$ denote the complex projective space.
Since this $2n$-dimensional manifold naturally arises as a coadjoint
orbit of $SU(n+1)$, it inherits a symplectic form $\omega$ and a
Hamiltonian $SU(n+1)$ action.

Thus, for any $j \in \{0,\dots,n-1\}$,
there is a semifree Hamiltonian  circle action given by
$$\lambda\cdot [z_0, z_1, ..., z_n]=[\lambda z_0, \lambda z_1, ...,
\lambda z_j, z_{j+1}, ..., z_{n}].$$
The fixed set $(\CP^n)^{S^1}$ consists of two components:
\begin{align*}
\left\{ [z] \in \CP^n
\mid z_k = 0 \ \forall \ k  \leq j \right\} &\simeq \CP^{n-j-1}
\quad  \mbox{and} \\
\left\{ [z] \in \CP^n
\mid z_k = 0 \ \forall \ k  > j \right\} &\simeq \CP^{j}.
\end{align*}
Note that $2(n - j - 1) + 2 j = 2n - 2$, as required.
\end{example}

\begin{example}\label{ex2}
Given $n \geq 3$, let  $\Tilde G_2(\R^{n+2})$ denote the
Grassmannian of oriented two-planes in $\R^{n+2}$. Since this
$2n$-dimensional manifold naturally arises as a coadjoint orbit of
$SO(n+2)$, it inherits a symplectic form $\omega$ and a Hamiltonian
$SO(n+2)$ action.

Thus, if $n$ is odd, there is a Hamiltonian circle
action on $\Tilde G_2(\R^{n+2})$ induced by the action on
$\R^{n+2}\cong \R \times \C^{\frac{1}{2}(n+1)}$ given by
$$\lambda \cdot (t, z_1, ..., z_{\frac{1}{2}(n+1)})=
(t, \lambda z_1, ..., \lambda z_{\frac{1}{2}(n+1)}).$$ The fixed set
consists of two components, corresponding to the two orientations on
the real two-planes in $\P \big( \{0\} \times \C^{\frac{1}{2}(n+1)}
\big) \simeq \CP^{\frac{1}{2}(n-1)}$. Note that $
2\left(\frac{1}{2}(n-1) \right) + 2\left(\frac{1}{2}(n-1)\right) =
2n -2$, as required. Moreover, $\Tilde G_2(\R^{n+2})^{\Z_2}$  is the
set of planes which lie entirely within $\{0\} \times
\C^{\frac{1}{2}(n+1)}$. This  submanifold, which  is
symplectomorphic to $\Tilde G_2(\R^{n+1})$, has codimension $2$.
\end{example}

\begin{remark}
It is natural to wonder whether Examples \ref{ex1} and \ref{ex2} are the only
examples that satisfy the conditions of Theorems \ref{A} and \ref{B}.
Unfortunately, this question is very hard.
If $X$ or $Y$ is an isolated point, or if
$\dim(M) \leq  6$ and the action is semifree, then  --
up to equivariant symplectomorphism  --
there are no other examples  \cite{De, Go}.
Although there are serious technical difficulties,
it may be possible using techniques from \cite{Go, Mc}
to extend this -- but only for  manifolds with $\dim(M) \leq 6$.
However,  using our results,  the first author,  Olbermann,  and Stanley recently proved
that $M$  is simply connected.  Using this, they showed that --
up to equivariant diffeomorphism --  there are finitely many
such manifolds in each dimension \cite{LOS}.
\end{remark}

The Atiyah-Bott-Berline-Vergne
localization formula
is not sufficient for the work of this
paper. Instead, we work much more directly with the cohomology
ring and Chern classes of the reduced space itself. We believe that
this more direct approach will be vital for further progress.

The results in this paper suggest several paths for further
research. The most obvious is to study
the questions on page \pageref{q1}
when $M^{S^1}$ contains more than two connected components. For example,
in Proposition~\ref{implies}, we show that whenever $M$ has fixed
components of minimal dimension, then the minimal fixed component
$X$ is a real cohomology projective space, that is, $H^*(X;\R)
\simeq H^*(\CP^k;\R)$ for some $k$.

Alternatively, one could attempt to classify circle actions with
two fixed components.
For example, in the
appendix, we show that for any effective Hamiltonian circle action
on compact symplectic manifolds with two fixed components, no point
has stabilizer $\Z_k$ for any $k
> 6$.
This raises another question.

\begin{question}
Does there exist  an effective  Hamiltonian circle action on a
compact symplectic manifold $(M,\omega)$ such that $M^{S^1}$ has
exactly two components and there exists a  point $x \in M$ with
stabilizer $\Z_k$, for $k >2$?
\end{question}

Hausmann and Holm  also studied  Hamiltonian
circle actions with two fixed components, but from a very different
perspective \cite{HH}.

The outline of this paper is straightforward. We describe properties of
moment maps in \S2; this is mostly review. In \S3, we use
Theorem~\ref{B} to prove Theorem~\ref{A} and calculate the equivariant cohomology of $M$.
 The rest of the paper is
dedicated to proving Theorem 2. We consider the implications of our
two  main restriction -- that the fixed components have minimal
dimension, and that there are only two fixed components, in \S4 and
\S5, respectively. In \S6, we bring these arguments together to
complete the proof of Theorem~\ref{B} in the semifree case. Finally,
in \S7 we study isotropy submanifolds of actions with only two fixed
components, and in \S8 we use this to complete the proof of
Theorem~\ref{B}.

\subsubsection*{Acknowledgment}
The first author  would like to thank  the University of Luxembourg,
and particularly the University of Illinois at Urbana-Champaign  for
financial support while she was visiting the second author. We  thank Professor Ono for suggesting several references.

\section{Background}\label{background}

The main goal of  this section is to introduce some background material and
establish our notation.
However,  in a few cases we will need to slightly extend
known results.

Let the circle act (possibly trivially) on a space $X$. The {\bf
equivariant cohomology} of $X$ is $$H^*_{S^1}(X) = H^*(X
\times_{S^1} S^\infty).$$ For example, if $p$ is a point then
$H^*_{S^1}(p;\Z) = H^*(\CP^\infty;\Z) = \Z[t]$. More generally, if
$F \subset X^{S^1}$ is a fixed component, then $H_{S^1}^*(F)$ is
naturally isomorphic to $H^*(F) \otimes H^*(\CP^\infty)= H^*(F)[t].$
In contrast, if the stabilizer of every point $x \in X$ is finite,
then $H_{S^1}^*(X;\R)$ is naturally isomorphic to $H^*(X/S^1;\R)$.
Note that the projection map $X\times_{S^1} S^{\infty}\rightarrow
\CP^{\infty}$ induces a pull-back map
\begin{equation}\label{pi}
\pi^* \colon H^*(\CP^{\infty})\rightarrow
H^*_{S^1}(X);
\end{equation}
hence, $H^*_{S^1}(X)$ is a $H^*(\CP^{\infty})$ module.

Let $M$ be a compact manifold.
A {\bf symplectic form} on $M$ is a closed,
non-degenerate two-form $\omega \in \Omega^2(M)$.
A circle action on $M$ is {\bf symplectic}
if it preserves $\omega$.  A symplectic circle
action is {\bf Hamiltonian} if there exists
a {\bf moment map}, that is,
a map $\phi \colon M \to \R$ such that
$$- d\phi = \iota_{\xi_M} \omega,$$
where $\xi_M$ is the vector field on $M$ induced by the circle
action. Since $\iota_{\xi_M} \omega$ is closed, every symplectic
action is Hamiltonian if $H^1(M;\R) = 0$.

As we mentioned in the introduction, the moment map $\phi \colon M
\to \R$ is a Morse-Bott function whose critical set is exactly the
fixed point set $M^{S^1}$. Therefore, if $c \in \R$ is a regular
value of $\phi$, then every point in the level set $\phi^{-1}(c)$
has finite stabilizer.  Since $\phi^{-1}(c)$ is a manifold, this
implies that the {\bf symplectic reduction} of $M$ at $c$,
$$M_c := \phi^{-1}(c)/S^1,$$
is an orbifold.
Since $\omega|_{\phi^{-1}(c)}$ is a basic two-form,
there exists a symplectic form $\omega_c \in \Omega^2(M_c)$ such
that $\rho^*(\omega_c) = \omega|_{\phi^{-1}(c)}$,
where $\rho \colon \phi^{-1}(c) \to M_c$ is the quotient map.
Let
$$\kappa \colon H_{S^1}^*(M;\R) \to H^*(M_c;\R)$$
be the composition of the restriction map from
$H_{S^1}^*(M;\R)$ to $H_{S^1}^*(\phi^{-1}(c);\R)$ and
the isomorphism from $H_{S^1}^*(\phi^{-1}(c);\R) $ to $H^*(M_c;\R)$;
this is called the {\bf Kirwan map}.

Given a symplectic manifold $(M,\omega)$, there is an almost complex
structure $J \colon T(M) \to T(M)$ which is {\bf compatible} with
$\omega$, that is, $\omega(J(\cdot), \cdot)$ is a Riemannian metric.
Moreover, the set of such structures is connected, and so there is a
well defined total  Chern class $c(M) \in H^*(M;\Z)$. Similarly,
given a circle action on $(M,\omega)$ with moment map $\phi \colon M
\to \R$, there is a well-defined multiset of integers, called {\bf
weights}, associated to each fixed point $p$.
In fact, for any fixed
component $F$, the tangent bundle $T(M)|_F$ naturally splits into
subbundles -- one corresponding to each weight.

The negative normal bundle $N_F^-$ at  a fixed component $F$ is the
sum of the subbundles of $T(M)|_F$ with negative weights.  In
particular, if  $\lambda_F$ is the number of negative weights in
$T_p M$ for any $p\in F$ (counted with multiplicity), then the {\bf
index} of $\phi$ as a Morse-Bott function  at $F$ is $2\lambda_F$.
Under the identification $H^*_{S^1}(F) = H^*(F)[t]$, the equivariant
Euler class $e^{S^1}(N_F^-)$  is a polynomial in $t$ whose highest
degree term is $\Lambda^-_F\, t^{\lambda_F}$, where $\Lambda_F^- \in
\Z \smallsetminus \{0\}$ is the product of the negative (integer)
weights at $F$. As Atiyah and Bott pointed out, this fact implies
that $e^{S^1}(N_F^-)$ is not a zero divisor in $H^{*}_{S^1}(F;\R)$.
Kirwan proved that this fact has remarkable consequences for Hamiltonian
actions \cite{K}; we explain some of these consequences below.

Let $R$ be a commutative ring (with unit), for
example, $\R$,  $\Z$, or $\Z_p$. If $R = \R$, or if the action is
semifree, or if $H^{j-2\lambda_F} (F;\Z)$ is torsion-free {\em and}
$R = \Z$,  then multiplication by $e^{S^1}(N_F^-)$ induces an
injection from $H^{j-2\lambda_F}_{S^1}(F;R)$ to $H^{j}_{S^1}(F;R)$.
(See \cite{TW} for comments on the case $R \neq \R$.)
Let $M^\pm = \phi^{-1}(-\infty, \phi(F) \pm \epsilon)$, where
$\epsilon > 0$ is sufficiently small. For simplicity, assume that
$F$ is the only fixed set in $M^+ \smallsetminus M^-$. By the
previous paragraph, if $R = \R$, or if the action is semifree, or if
$H^{j-2\lambda_F}(F;\Z)$ and $H^{j - 2\lambda_F +1}(F;\Z)$ are
torsion-free and $R = \Z$, then
 the map from $H_{S^1}^{*}(M^+,M^-;R)=H_{S^1}^{*-2\lambda_F}(F;R)$ to $H_{S^1}^*(F;R)$
 is injective for $*=j$ and $*=j+1$;
therefore, the long exact sequence in equivariant
cohomology for the pair $(M^+,M^-)$ breaks into a short exact
sequence
\begin{equation}
\label{short}
 0 \to H_{S^1}^j(M^+,M^-;R) \to H_{S^1}^j(M^+;R)
\to H_{S^1}^j(M^-;R) \to 0.
\end{equation}
In particular, if $j \leq 2 \lambda_F$ and $R = \Z$, then
\eqref{short} is exact because $H^i(F;\Z)$ is torsion-free for all
$i \leq 1$. (Note that, if  $j \leq 2 \lambda_F - 2$ then
$H_{S^1}^j(M^+,M^-;R) = H_{S^1}^{j+1}(M^+,M^-;R) = 0$ and so
$H^j_{S^1}(M^+;R) = H^j_{S^1}(M^-;R)$ for any commutative ring $R$.)

Additionally, restriction induces a natural map of exact sequences
\begin{equation*}
\minCDarrowwidth20pt
\begin{CD}
\dots  @>>> H^j_{S^1}(M^+,M^-;R) @>>> H^j_{S^1}(M^+;R)
  @>>> H^j_{S^1}(M^-;R)  @>>>  \dots \\
@. @VVV @VVV @VVV @. \\
\dots  @>>> H^j(M^+,M^-;R) @>>> H^j(M^+;R)   @>>> H^j(M^-;R)  @>>> \dots.  \\
\end {CD}
\end{equation*}
The restriction map from $H^j_{S^1}(M^+,M^-;R)$ to $H^j(M^+,M^-;R)$
is surjective because $H^j_{S^1}(F;R) = H^j(F;R)[t]$.
Hence, by an easy diagram chase,
if \eqref{short} is exact
and
if the restriction map from $H^j_{S^1}(M^-;R)$ to $H^j(M^-;R)$ is surjective,
then
the restriction map from $H^j_{S^1}(M^+;R)$ to $H^j(M^+;R)$ is also surjective.
If, additionally,
the restriction map from $H^{j-1}_{S^1}(M^-;R)$ to $H^{j-1}(M^-;R)$ is
surjective,
then
\begin{equation}\label{short2}
 0 \to H^j(M^+,M^-;R) \to H^j(M^+;R)
\to H^j(M^-;R) \to 0
\end{equation}
is a short exact sequence.

Note that, if $j = 2$ and $R = \Z$,  then \eqref{short} is exact for
any fixed component $F$. This is because either  $\lambda_F = 0$, in
which case $e^{S^1}(N_F^-) = 1$, or $\lambda_F > 0$, in which
case $H^{j - 2 \lambda_F}(F;\Z)$ and $H^{j-2\lambda_F+1}(F;\Z)$ are
torsion-free. Therefore, the proposition below follows easily by
induction and the paragraph above.

\begin{proposition}\label{formal}
Let the circle act
on a compact symplectic manifold
$(M,\omega)$ with moment map $\phi \colon M \to \R$. The natural
restriction $H^2_{S^1}(M;\Z) \to H^2(M;\Z)$ is onto.
\end{proposition}

More generally, as we saw above, if the action is semifree, if $R =
\R$, or if $H^*(M^{S^1};\Z)$ is torsion-free and $R = \Z$,  then
\eqref{short} is exact for every $j$. By induction, this implies
that \eqref{short2} is exact for all $F$, that the restriction map
$H^*_{S^1}(M;R) \to H^*(M;R)$ is a surjection, and that the
restriction map $\iota^* \colon H^*_{S^1}(M;R) \to
H^*_{S^1}(M^{S^1};R)$ is an injection.

If $H^*(M^{S^1};\Z)$ is torsion-free and $R = \Z$, or if $R = \R$,
then \eqref{short} and \eqref{short2}
imply that
\begin{gather*}
 H_{S^1}^j(M;R) = \bigoplus_{F \subset M^{S^1}} H_{S^1}^{j - 2\lambda_F}(F;R),
\quad \mbox{and}  \\
H^j(M; R) = \bigoplus_{F \subset M^{S^1}} H^{j - 2 \lambda_F}(F;R)
\quad \forall j,
\end{gather*}
where the sum is over all fixed components. In particular, $\phi$ is
{\bf perfect} and {\bf equivariantly perfect}, that is, these
equations hold for $R = \R$.

In particular, if $H^*(M^{S^1};\Z)$ is torsion-free and $R = \Z$, or if
$R = \R$, then $H^*(M;R)$ is a finitely generated free $R$-module and
 the restriction map from $H^j_{S^1}(M;R) $ to $H^j(M;R)$ is surjective.
Therefore,
applying the Leray-Hirsch Theorem to the fiber
bundle $M\times_{S^1} S^\infty \stackrel{\pi}{\to} \CP^\infty$, we
see that
the kernel of  the restriction map $ H^j_{S^1}(M;R) \to  H^j(M;R) $  is
the ideal generated by $\pi^*(t)$, where $t \in
H^2(\CP^\infty;R)$ is the generator.
Hence, if we
want to compute the (ordinary) cohomology of $M$, it is enough to
determine the equivariant cohomology of $M$;
\begin{equation}
\label{ordequi} H^*(M;R) = H^*_{S^1}(M;R)/(t).
\end{equation}

We can use a similar argument as in \cite{T} to prove the following claim.

\begin{proposition}\label{multiple-Euler}
Let the circle act on a compact symplectic manifold $(M, \omega)$
with moment map $\phi \colon M \to \R$. Consider $\beta\in
H^*_{S^1}(M; \R)$
so that $\beta|_{F'}=0$ for all fixed components $F'$ such
that $\phi(F') < \phi (F)$, where $F$ is a fixed component. Then
$\beta|_F$ is a multiple of $e^{S^1}(N_F^-)$.
\end{proposition}

Given any manifold $M$, there is a natural map
from $H^*(M;\Z)$ to $H^*(M;\R)$.
The image of this map is a lattice in $H^*(M;\R)$.
We shall say that a class is {\bf integral} if it lies in
the image of this map and is {\bf primitive} if,
in addition,
it is not a positive integer multiple of any other integral class.

\begin{lemma}\label{uvt}
Let the circle act  on a compact symplectic manifold
$(M,\omega)$ with moment map $\phi \colon M \to \R$.  Let $F$ be a
fixed component.
\begin{itemize}
\item There exists
$\ut \in H_{S^1}^2(M;\R)$ so that,
\begin{gather*}
\ut|_{F'} = [\omega|_{F'}] + t \left(\phi(F) - \phi(F') \right)
\quad \mbox{and} \\
\kappa_c\big(\ut - t (\phi(F) - c)\big) = \omega_c
\end{gather*}
for all fixed components $F'$ and all regular values $c \in \R$.
Here, $\kappa_c \colon H^*_{S^1}(M;\R) \to H^*(M_c; \R)$ is the
Kirwan map and $(M_c,\omega_c)$ is the symplectic reduction of $M$
at $c$.
\item If $[\omega]$ is integral, then  $\ut$ is integral.
\end{itemize}
\end{lemma}

\begin{proof}
To prove the first claim, let  $\ut = [\omega - t \phi +  t\phi(F)
]$ in the De Rham model of equivariant cohomology. If $c$ is a
regular value, then $\ut -  t\phi(F) + tc = [\omega - t \phi + t
c]$, and so $(\ut - t \phi(F) + tc)|_{\phi^{-1}(c)} =
[\omega|_{\phi^{-1}(c)}]$, which maps to $\omega_c$ under the
natural isomorphism $H_{S^1}^*(\phi^{-1}(c);\R) \simeq H^*(M_c;
\R)$.

If $[\omega] \in H^2(M;\R)$ is integral, then by Proposition~\ref{formal},
there exists an integral class
$\at \in H^2_{S^1}(M;\R)$ which maps to $[\omega]$
under the natural restriction $H_{S^1}^2(M;\R) \to H^2(M;\R)$.
Then the
restriction of $\at - \ut$ under the same map is zero.
Since $H^*(\CP^\infty;\R) = \R[t]$,  this
implies that $\at - \ut = \lambda t$ for some constant
$\lambda\in \R$. Finally, since $\at$ is integral,
$\at|_F = \ut|_F +\lambda t = [\omega|_F] +
\lambda t$ is integral; hence $\lambda\in\Z$.
\end{proof}

\begin{lemma}\label{Chern}
Let the circle act on a complex vector bundle $E$ of rank $d$ over a
compact manifold $X$ so that $E^{S^1} = X$. Assume that there exists
a non-zero $\lambda \in \Z$ so that the circle acts on $E$ with
weight $\lambda$. Then there exists $c_i \in H^{2i}(X;\Z)$ for all
$i \in \left\{0,\dots,\frac{1}{2}\dim(M) \right\}$ such that
\begin{align*}
c(E) &= 1 + c_1 + \dots + c_{d-1} + c_d, \\
c^{S^1}(E) &= (1 + \lambda t)^d + c_1 (1 + \lambda t)^{d-1}+ \dots + c_{d-1} (1 + \lambda t) +c_d,
\quad \mbox{and} \\
e^{S^1}(E) &=
(\lambda t)^d + c_1 (\lambda t)^{d-1}+ \dots +  c_{d-1} (\lambda t)
+c_d.
\end{align*}
Here, $c(E)$, $c^{S^1}(E)$, and $e^{S^1}(E)$ are the total Chern
class of $E$, the total {\em equivariant} Chern class of $E$, and
the  equivariant Euler class of $E$, respectively.
\end{lemma}

\begin{proof}
By the splitting principle, there exists a space $Y$
and a map $p \colon Y \to X$  such that $p^* \colon H^*(X;\Z) \to H^*(Y;\Z)$
is injective and the pullback bundle $p^*(E)$ breaks up as the direct
sum of line bundles.
Therefore, without loss of generality we  may assume that the bundle is a direct sum of line bundles
with first Chern class $\alpha_1,\dots \alpha_d$ respectively.
Then
$$
c(E) = \prod_{i=1}^d (1 + \alpha_i)
\quad \mbox{and} \quad
c^{S^1}(E) = \prod_{i=1}^d (1  + \lambda t + \alpha_i).
$$
The claim follows immediately.
\end{proof}

\section{Using Theorem~\ref{B} to calculate the cohomology ring of $M$}
\label{sec:calc}

In this section, we use Theorem~\ref{B} to compute the possible
cohomology rings (ordinary and equivariant) and Chern classes of
$M$. More specifically, we prove Theorem~\ref{A} and
Proposition~\ref{C}. We give two proofs. The first is very short but
depends on knowing the cohomology ring and  Chern classes of $\CP^n$
and $\Gt_2(\R^{n+2})$. The second relies on a direct calculation.

 The cohomology ring and  Chern classes of $\CP^n$ and (if $n \neq
1$ is odd) of $\Gt_2(\R^{n+2})$ are exactly those described in
Theorem~\ref{A}. Moreover, the fixed set data which arises in
Examples~\ref{ex1} and \ref{ex2} is exactly the fixed set data
described in Theorem~\ref{B}. Finally, we can transform any
non-trivial circle action into an effective circle action by
quotienting out by the subgroup which acts trivially. Therefore,
Thereom~\ref{A} is an immediate consequence of Theorem~\ref{B} and
the proposition below, which combines Corollary 3.16 and Remark 3.18
in \cite{T}.

\begin{proposition}[Tolman]\label{Tolman}
Let the circle act on  compact symplectic manifolds $(M,\omega)$ and
$(\Mh,\oh)$ with moment maps $\phi \colon M \to \R$ and $\ph \colon
\Mh \to \R$, respectively; assume that $H^j(\Mh;\Z) = H^j(\CP^n;\Z)$
for all $j$. Also assume that there exists a bijection from the
fixed components $F_1,\dots,F_k$ of $M$ to the fixed components
$\Fh_1,\dots,\Fh_k$ of $\Mh$ so  that there exists an isomorphism
$f^* \colon H_{S^1}^*(\Fh_i;\Z) \to H_{S^1}^*(F_i;\Z)$ such that
$f^*(c^{S^1}(\Mh)|_{\Fh_i}) = c^{S^1}(M)|_{F_i}$ for all $i$. Then
there is an isomorphism $f^\sharp \colon H^*(\Mh;\Z) \to H^*(M;\Z)$
so that $f^\sharp(c(\Mh)) = c(M)$.
\end{proposition}

The argument that the {\em equivariant} cohomology and Chern classes
of $M$ agree with those of $\CP^n$ or $\Gt_2(\R^{n+2})$ is nearly
identical, except that \cite[Corollary 3.16]{T} should be replaced
by \cite[Corollary 3.13]{T}.\\

 Alternatively, we can use
Theorem~\ref{B} to directly calculate the (equivariant) cohomology
and Chern classes of $M$.

\begin{proposition}\label{C}
Let the circle act effectively on a   compact symplectic
$2n$-dimensional manifold $(M, \omega)$  with moment map $\phi
\colon M \to \R$. Assume that  $M^{S^1}$ has exactly two components,
$X$ and $Y$, and that $\dim(X) + \dim(Y) = \dim(M) - 2$. If $\phi(X) < \phi(Y)$,
then one of
the following is true:
\begin{itemize}
\item[(A)]
\qquad \quad $H^*_{S^1}(M;\Z) = \Z[\xt, t]/ \big(
\xt^{i+1}\big(\xt+t\big)^{j+1} \big) \quad \mbox{and}$
$$c^{S^1}(M) = \big( 1 + \xt\big)^{i + 1} \big(1 +  \xt + t\big)^{j +
1}, \quad \mbox{where}$$
$$\xt|_X = u, \quad\mbox{and} \quad \xt|_ Y = v- t.$$

\item[(B)]
$H^*_{S^1}(M;\Z) = \Z[\xt,\yt, t]/ \Big(  \xt^{i+1} - 2 \yt, \yt
\big(\yt + \textstyle\frac{1}{2} \big( (\xt+2t)^{i+1} - \xt^{i+1}
\big) \big) \Big) \quad \mbox{and}$
$$c^{S^1}(M) = \frac {( 1 + \xt )^{i+1} (1 + \xt + 2t)^{i+1} (1 + \xt
+ t)} {1 + 2 \xt + 2 t}, \quad  \mbox{where}$$
$$\xt\big|_X = u, \quad \xt\big|_Y = v- 2t, \quad \yt\big|_X = 0,
\quad \mbox{and} \quad \yt\big|_Y = \textstyle\frac{1}{2} (v
-2t)^{i+1}.$$
\end{itemize}
In both cases, $t$ generates
$\pi^*\left(H^2(\CP^\infty;\Z)\right) \subset H^2_{S^1}(M;\Z)$ and
$\xt$ has degree $2$; in case (B), $\yt$ has degree $n+1$.
Moreover, $u$ and $v$ are the positive generators of
$H^2(X;\Z)$ and $H^2(Y;\Z)$, respectively, $\dim(X) = 2i$, and $\dim(Y) = 2j.$
\end{proposition}

\begin{proof}
By  Theorem~\ref{B},
$H^*(X;\Z) = \Z[u]/u^{i+1}$ and $ H^*(Y;\Z) =
\Z[v]/v^{j+1}.$
In particular, $H^*(M^{S^1};\Z)$ is torsion-free.
As we showed in \S\ref{background}, this implies that
the restriction map $H_{S^1}^*(M;\Z) \to H_{S^1}^*(M^{S^1};\Z)$ is injective and  that
$$H^{k}(M;\Z) = \bigoplus_{F \subset M^{S^1}} H^{k - 2\lambda_F}(F;\Z)
= H^k(\CP^n;\Z) \quad \forall \ k.$$

Proposition 3.9 of \cite{T} states that whenever $H^k(M;\Z) = H^k(\CP^n;\Z)$
for all $k$,
the classes $1,\alpha_1,\dots,\alpha_n$ defined by
\begin{equation}\label{alphai}
\alpha_k = \frac{\Lambda^-_{F_k}}{m_k} \big(c_1^{S^1}(M) -
\Gamma_{F_k} t\big)^{k - \lambda_{F_k}} \prod_{\lambda_{F'} <
\lambda_{F_k}} \left(\frac{ c_1^{S^1}(M) - \Gamma_{F'}
t}{\Gamma_{F_k} - \Gamma_{F'} } \right)^{\frac{1}{2}\dim(F') + 1}
\end{equation}
form a basis for $H_{S^1}^*(M;\Z)$ as a $H^*(\CP^\infty;\Z) = \Z[t]$
module. Here, $F_k$ is the unique fixed component so that $H^{2k - 2
\lambda_{F_k}}(F_k;\Z) = \Z$, and $m_k \in \Z$ is chosen so that
$\frac{1}{m_k} c_1(M)^{k - \lambda_{F_k}}\big|_{F_k}$ generates
$H^{2k-2\lambda_{F_k}}(F_k;\Z)$ for each integer $k$ such that $0
\leq 2k \leq 2n$. Moreover, $\Lambda^-_{F}$ is the product of the
negative weights at $F$ and $\Gamma_F$ is the sum of the weights at
$F$ for each fixed component $F$. Finally, the product is over all
fixed components $F'$ such that $\lambda_{F'}  < \lambda_{F_k}.$

Assume first that the action is
semifree. By part (A) of Theorem~\ref{B} and Lemma~\ref{Chern},
\begin{equation}\label{CA}
c_1^{S^1}(M)\big|_X=(n+1)u+(j+1)t \quad\mbox{and} \quad  c_1^{S^1}(M)\big|_Y=(n+1)v-(i+1)t.
\end{equation}
Hence, $\frac{1}{(n+1)^k} c_1(M)^k\big|_X$ generates $H^{2k}(X;\Z)$  for all $k \in \{0,\dots,i\}$;
similarly,
$\frac{1}{(n+1)^k} c_1(M)^k\big|_Y$ generates $H^{2k}(Y;\Z)$ for all $k \in \{0,\dots,j\}$.
Additionally, $\Gamma_X = (j+1) $, $\Gamma_Y = - (i+1)$,
$\Lambda^-_X = 1$,  and $\Lambda^-_Y =   (-1)^{i+1}$.
Therefore, by \eqref{alphai},
$$\alpha_k =
\begin{cases}
\frac{1}{(n+1)^k}
\big( c_1^{S^1}(M) -( j + 1 ) t  \big)^k
& 0 \leq k \leq i \\
\frac{1}{ \left(n + 1 \right)^k}
\big( c_1^{S^1}(M) - \left( j + 1 \right) t \big)^{i+1}
\big( c_1^{S^1}(M) + \left( i + 1 \right) t \big)^{k-i-1}
& i < k \leq n.
\end{cases}
$$
In particular, \eqref{CA} implies that $\alpha_1|_X = u$ and
$\alpha_1|_Y = v - t$.  Hence, if we let $\xt=\alpha_1$, then
$\xt^{i+1}\big(\xt+t\big)^{j+1}=0$  and    part (A) of
Theorem~\ref{B} and Lemma~\ref{Chern} imply that
\begin{align*}
c^{S^1}(M)\big|_X  &= ( 1 + \xt)^{i + 1} (1 +  \xt + t)^{j + 1}\big|_X
\quad  \mbox{and} \\
c^{S^1}(M)\big|_Y  &= ( 1 + \xt)^{i + 1} (1 +  \xt + t)^{j + 1}\big|_Y.
\end{align*}
Since the restriction map $H^*_{S^1}(M;\Z) \to H_{S^1}^*(M^{S^1};\Z)$
is injective,  claim (A) follows easily.

Now assume that the
action is not semifree. By part (B) of Theorem~\ref{B} and
Lemma~\ref{Chern}, $\dim(X) = \dim(Y)$ and so $i =
\frac{1}{2}(n-1)$; moreover,
\begin{equation}\label{CB}
c_1^{S^1}(M)\big|_X = nu+nt \quad\mbox{and} \quad
c_1^{S^1}(M)\big|_Y = nv-nt.
\end{equation}
Hence, $\frac{1}{n^k} c_1(M)^k\big|_X$ generates $H^{2k}(X;\Z)$ and
$\frac{1}{n^k} c_1(M)^k\big|_Y $ generates $H^{2k}(Y;\Z)$
for all $k \in \{0,\dots,i\}$.
Additionally, $\Gamma_X = n$, $\Gamma_Y = - n$,
$\Lambda^-_X = 1$, and $\Lambda^-_Y =  2^i (-1)^{i+1}$.
Therefore, by \eqref{alphai},
$$\alpha_k =
\begin{cases}
\frac{1}{n^k}
\big( c_1^{S^1}(M) -n t  \big)^k
& 0 \leq k \leq i \\
\frac{1}{2 n^k}
\big( c_1^{S^1}(M) - n t \big)^{i+1}
\big( c_1^{S^1}(M) + n t \big)^{k-i-1}
& i < k \leq n.
\end{cases}
$$
In particular, \eqref{CB} implies that $\alpha_1|_X = u$,
$\alpha_1|_Y = v - 2t$, $\alpha_{i+1}|_X = 0$, and $\alpha_{i+1}|_Y
= \frac{1}{2}(v-2t)^{i+1}$. Hence, if we let $\xt = \alpha_1$ and
$\yt = \alpha_{i+1}$, then $\xt^{i+1} - 2 \yt=0$ and $ \yt \big(\yt
+ \textstyle\frac{1}{2} \big( (\xt+2t)^{i+1} - \xt^{i+1} \big) \big)
= 0$. (Note that the latter expression does lie in $\Z[\xt,\yt,t]$,
while the expression $\frac{1}{2}\yt(\xt+2t)^{i+1}$ does not.)
Moreover, part (B) of Theorem~\ref{B} and Lemma~\ref{Chern} imply
that
\begin{align*}
c^{S^1}(M)\big|_X &= \frac{ (1 + \xt)^{i + 1}
(1 + \xt + 2t)^{i + 1} (1 + \xt + t)}
{1 + 2 \xt + 2 t}\Big|_X ,
\quad  \mbox{and} \\
c^{S^1}(M)\big|_Y &= \frac{ (1 + \xt)^{i + 1}
(1 + \xt + 2t)^{i + 1} (1 + \xt + t)}
{1 + 2 \xt + 2 t}\Big|_Y .
\end{align*}
Since the restriction map  $H^*_{S^1}(M;\Z) \to H_{S^1}^*(M^{S^1};\Z)$ is injective, claim (B) follows
easily.

\end{proof}

Finally, as we showed in \S\ref{background}, the fact that
$H^*(M^{S^1};\Z)$ is torsion-free implies that $H^*(M;\Z) =
H_{S^1}^*(M;\Z)/(t)$, where $t$ generates
the image
$\pi^*(H^2(\CP^\infty;\Z)).$ (See \eqref{ordequi}.) Therefore,
Theorem~\ref{A} follows immediately from the proposition above.

\section{The case that the fixed components have minimal dimension}

Consider a  Hamiltonian circle action on a compact symplectic
manifold $(M,\omega)$. In this section, we consider the case that
the fixed components have minimal dimension,
that is,
\begin{equation}
\label{eqdim}
\sum_{F \subset M^{S^1}} \left( \dim(F)  + 2 \right) =
\dim(M) + 2.
\end{equation}

We first prove that \eqref{eqdim} holds whenever $M$ has {\bf
minimal even Betti numbers}, that is,
\begin{equation}
\label{eqbetti}
b_{2i}(M) = 1 \quad \forall \ i \in \big\{ 0,\dots,\textstyle\frac{1}{2} \dim(M) \big\},
\end{equation}
where $b_j = \dim \big( H^j(M;\R) \big)$ for all $j$.

\begin{lemma}\label{index}
Let the circle act on a compact symplectic manifold $(M, \omega)$
with moment map $\phi \colon M \to \R$. Assume that $b_{2i}(M) = 1$
for all $i \in \left\{0,\dots,\tfrac{1}{2}\dim(M)\right\}.$
Then
$\sum_F \left( \dim(F) + 2 \right) = \dim(M) + 2.$
\end{lemma}

\begin{proof}
This claim is an immediate consequence of Lemma 3.3 in \cite{T},
which states that, for each $i \in \left\{0,\dots,\frac{1}{2}\dim(M)\right\}$,
there exists a unique fixed component $F$ such
that $0 \leq 2i - 2 \lambda_F  \leq \dim(F).$
(Lemma 3.3 itself follows from the
facts that $\phi$ is a perfect Morse-Bott function
and that $H^{2i}(F;\R) \neq 0$ for all
fixed components $F$ and $i \in \left\{0,\dots,\frac{1}{2} \dim(F)\right\})$.
\end{proof}

The following proposition -- which is the main result in this section --
gives a partial converse.

\begin{proposition}\label{implies}
Let the circle act on a compact symplectic manifold $(M, \omega)$
with moment map $\phi \colon M \to \R$.
Assume that $\sum_{F \subset M^{S^1}}
\left( \dim(F) + 2 \right) = \dim(M) + 2$.
Let $X$ be the minimal fixed component.
\begin{enumerate}
\item $H^i(M;\R) = H^i(\CP^n;\R)$ for all $i \in \{0,\dots,\dim(X) + 2\}$.
\item
$H^{*}(X;\R) = \R[u]/u^{\frac{1}{2}\dim(X) + 1}$, where $u = [\omega|_X]$.
\item  If $[\omega]$ is integral and
the integers $\{ \phi(X) - \phi(F)\}_{F \subset (M\smallsetminus X)^{S^1}} $ are relatively prime\footnote{
A set of integers is {\bf relatively prime} if their greatest common divisor is $1$.},
then $H^*(X;\Z) = \Z[u]/u^{\frac{1}{2}\dim(X) + 1}$, where $u$ maps
to $[\omega|_X].$
\end{enumerate}
\end{proposition}

\begin{remark}\label{compare}
If $M$ is $6$-dimensional, then
\eqref{eqdim}
implies \eqref{eqbetti}.  The same claim holds if  $M$ is $8$-dimensional,
unless $M^{S^1}$ has exactly three components:
a  minimal point, a maximal point,
and a $4$-dimensional component of index $2$.
However, \eqref{eqdim} does not imply \eqref{eqbetti} in this case.

To see the first two claims, consider a Hamiltonian circle action
on a compact symplectic manifold $(M,\omega)$
which satisfies \eqref{eqdim}.
Clearly, $H^{0}(F;\R) = H^{\dim(F)}(F;\R) = \R$ for  every fixed component $F$.
Therefore, if every fixed component has dimension $0$
or $2$, then \eqref{eqbetti} follows from  Lemma~\ref{index'} below.
In the remaining cases considered above, \eqref{eqbetti} follows from Lemma~\ref{index'}
and Poincar\'e duality on $M$ and $F$.

To see the last claim,  note that for any
$n > 2$ there is a Hamiltonian circle action on $\Gt_2(\R^{n+2})$
induced by the action on $\R^{n+2} \simeq \C \times \R^n$
given by
$$\lambda \cdot (z,x_1,\dots,x_n) =
( \lambda z,x_1,\dots,x_n).$$
(See Example~\ref{ex2}.)
The fixed set has three components.
Two are isolated fixed points which correspond
to the orientations on the real two-plane $\C \times \{0\}$.
The third component  has dimension $2n - 4$ and
corresponds to the set of oriented real two-planes
in $\{ 0 \} \times \R^n$.
Hence, $(0+2) + (0+2) + (2n-4+2) = 2n+2$, as required
by \eqref{eqdim}.
However, if $n$ is even, then $H^{n}\big(\Gt_2(\R^{n+2});\R\big) = \R^2$.
\end{remark}

To prove Proposition~\ref{implies}, we will need the following
analog of Lemma 3.3 in \cite{T}.

\begin{lemma}\label{index'}
Let the circle act on a compact symplectic manifold $(M, \omega)$
with moment map $\phi \colon M \to \R$. Assume that $\sum_{F \subset
M^{S^1}} \left( \dim(F) + 2 \right) = \dim(M) + 2$.
\begin{itemize}
\item
For each $i \in \left\{0,\dots,\frac{1}{2}\dim(M)\right\}$,
there exists a unique fixed component $F$ such
that $0 \leq 2i - 2\lambda_F \leq \dim(F)$.
\item In particular, if $X$ is the minimal fixed component,
then $\dim(X) \leq 2\lambda_F - 2$
for all other fixed components $F$.
\end{itemize}
\end{lemma}

\begin{proof}
Since $M$ is symplectic $H^{2i}(M;\R) \neq 0$ for all $i \in
\{0,\dots,\frac{1}{2} \dim(M)\}$. Since $\phi$ is a  Morse-Bott
function,  there is at least one fixed component
$F$ such that $0 \leq 2i - 2\lambda_F \leq \dim(F)$. Since
$\sum_{F \subset M^{S^1}} \left( \dim(F) + 2 \right) = \dim(M) + 2$, this
proves the claim.
\end{proof}

\begin{remark}\label{extend}
Consider a Hamiltonian circle action on a compact
symplectic manifold $(M,\omega)$; assume that \eqref{eqdim} holds.
Although we will not use them in this paper,
several of the results in \S 3 of \cite{T} still work in this context
if we use Lemma~\ref{index'} above instead of Lemma 3.3 in \cite{T}.
For example,  the proof of Proposition 3.4 and Lemma 3.7 in \cite{T}
otherwise go through without any changes.
Therefore,
for all fixed components $F$ and $F'$,
\begin{gather*}
\phi(F') < \phi(F) \quad \mbox{exactly if} \quad \lambda_{F'} < \lambda_F;
\quad \mbox{moreover} \\\
H^j(M;\Z) = \bigoplus_{F \subset M^{S^1}} H^{j-2\lambda_F}(F;\Z) \quad \forall \ j,
\end{gather*}
where the sum is over all fixed components.
\end{remark}

\begin{lemma}\label{betti0}
Let the circle act on a compact symplectic manifold $(M, \omega)$
with moment map $\phi \colon M \to \R$.  Let $X$ be the minimal
fixed component and let $\F$ be a field. Assume that $\dim(X) \leq  2
\lambda_F- 2$ for all other fixed components $F$. Assume also that
there exist classes $\ut \in H^2_{S^1}(M; \F)$ and  $u \in H^2(X;
\F)$, such that $\ut|_X = u$,  and a fixed point $y$ such that
$\ut|_y  \neq 0$. Then
$$H^*(X;\F) = \F[u]/u^{\frac{1}{2} \dim(X)+1}.$$
\end{lemma}

\begin{proof}
Assume that, on the contrary, $H^*(X;\F) \neq
\F[u]/u^{\frac{1}{2} \dim(X)+1}$.

First, we claim that there exist $\alpha \in H^{j}(X;\F)$ and $\at
\in H^j_{S^1}(M;\F)$ such that $\alpha \neq 0$, $\at|_X = \alpha$,
and $\at|_y = 0$. To see this, note that at least
one of the following is true:
\begin{itemize}
\item [(a)] there exists a non-zero class $\alpha \in H^{2i+1}(X;\F)$
for some $i$; or
\item [(b)] there exists a class $\alpha \in H^{2i}(X;\F)$
which is not a multiple of $u^i$ for some $i$. (Since $X$ is symplectic;
$H^{2i}(X;\F) \neq 0$ for all $i \in \left\{0,\dots,\frac{1}{2}\dim(X) \right\}$.)
\end{itemize}
If (a) is true, then since $\dim(X) < 2 \lambda_F$ for all other
fixed components $F$, there exists a class $\at \in
H^{2i+1}_{S^1}(M;\F)$ such that $\at\big|_X = \alpha$.
(See \eqref{short}.) Since
$H^{2i+1}(\CP^\infty;\F) = 0$, $\at\big|_y = 0$. Similarly,
if (b) is true then there exists $\at \in H^{2i}_{S^1}(M; \F)$ such
that $\at\big|_X = \alpha$. Since $\ut^i\big|_y \neq 0$,
we can define $\lambda = \frac{\at|_y}{\ut^i|_y}$ and then replace
$\alpha$ by $\alpha - \lambda u^i$ and $\at$ by $\at - \lambda
\ut^i$.

Since $\F$ is a field, Poincar\'e duality implies that there exists
a class $\beta \in H^{\dim(X) - j}(X;\F)$ such that $\alpha \cup
\beta = u^{\frac{1}{2} \dim(X)}$.  As before, there exists $\bt \in
H_{S^1}^{\dim(X) - j}(M; \F)$ such that $\bt\big|_X = \beta$. Since
$\ut^{\frac{1}{2}\dim(X)}\big|_X =  \left( \at \cup \bt
\right)\big|_X$, and since $\dim(X) \leq  2 \lambda_F - 2$ for all
other fixed components $F$, we can conclude that
$\ut^{\frac{1}{2}\dim(X)} =   \at \cup \bt $.  But $ (\at \cup
\bt)\big|_y = \at\big|_y \cup \bt\big|_y  = 0$, while
$\ut^{\frac{1}{2}\dim(X)}\big|_y  = \big( \ut\big|_y
\big)^{\frac{1}{2}\dim(X)} \neq 0$. This gives a contradiction.
\end{proof}

We are now ready to prove our main result.

\begin{proof}[Proof of Proposition~\ref{implies}.]
By Lemma~\ref{index'}, $\dim(X) \leq 2\lambda_F - 2$ for every other fixed
component $F$.
Moreover, there is exactly one fixed component $F$ with $2\lambda_F = \dim(X) + 2$.

By Lemma~\ref{uvt}, there exists $\ut \in H^2_{S^1}(M;\R)$
such that $\ut|_X = [\omega|_X]$ and $\ut|_y = t (\phi(X) - \phi(y)) \neq 0$
for all fixed points $y \not\in X$.
Since $\phi$ is a perfect Morse-Bott function,
claims (1) and  (2) are an immediate consequence of Lemma~\ref{betti0}.

If $[\omega]$ is integral, then by Lemma~\ref{uvt}
there exists
$\ut \in H^2_{S^1}(M;\Z)$ and $u \in H^2(X;\Z)$
so that $\ut|_X = u$, $\ut|_y = t(\phi(X) - \phi(y))$ for all fixed
points $y$, and $u$ maps to $[\omega|_X] \in H^2(X;\R)$.
If the integers $\{\phi(X) - \phi(F)\}_{F \subset (M\smallsetminus X)^{S^1}}$ are relatively prime,
then for any prime $p$ there exists a fixed point $y$ so
that $\phi(X) - \phi(y) \neq 0 \mod p$.
Therefore, by Lemma~\ref{betti0},
$H^*(X;\Z_p) = \Z_p[u]/u^{\frac{1}{2}\dim(X) + 1}$.
On the one hand, by Lemma~\ref{torsion} below, this implies that
$H^*(X;\Z)$ is torsion-free.
On the other hand, it implies that $u^i$ is primitive
for all $i \in \{0,\dots,\frac{1}{2}\dim(X)\}$.
Claim (3) follows immediately.
\end{proof}

\begin{lemma}\label{torsion}
Let $X$ be a compact manifold.
Assume that $H^{2i+1}(X;\Z_p) = 0$ for all $i$ and all primes $p$.
Then $H^*(X;\Z)$ is torsion free.
\end{lemma}

\begin{proof}
 Since $X$ is compact, the homology ring of $X$ is
finitely generated. Moreover, $\Hom(\Z,\Z_p) = \Z_p$ for all primes
$p$, while $\Hom(\Z_q,\Z_p) = \Z_p$ and $\Ext(\Z_q,\Z_p) = \Z_p$  if
$p$ divides $q$. Therefore, the claim follows immediately from the
universal coefficient theorem.
\end{proof}

\begin{lemma}\label{euler}
Let the circle act on a compact connected symplectic manifold $(M, \omega)$
with moment map $\phi \colon M \to \R$.
Let $X$ be the minimal fixed component.
Let $e^{S^1}(N_X) \in H^{\dim(M) - \dim(X)}_{S^1}(X;\R)$ be the
equivariant Euler class of the
normal bundle of
$X$, and let $\Lambda_X$ be the product of the weights (with multiplicity) on
the normal bundle of $X$.
\begin{enumerate}
\item
If $\sum_{F \subset M^{S^1}} \left (\dim(F) + 2 \right) = \dim(M) + 2,$
then
\begin{equation*}
 e^{S^1}(N_X) = \Lambda_X
\negthickspace \negthickspace \negthickspace
 \negthickspace
 \prod_{F \subset
\left( M \smallsetminus X\right)^{S^1}}
 \negthickspace
\left ( t + \frac{[\omega|_X] }{\phi(F) - \phi(X)}
\right )^{\frac{1}{2}\dim(F) + 1}
\negthickspace \negthickspace \negthickspace
\negthickspace \negthickspace \negthickspace
 \negthickspace \negthickspace
\negthickspace \negthickspace \negthickspace
\negthickspace \negthickspace \negthickspace
,
\end{equation*}
where
the product is over all fixed components except $X$.
\item  More generally, there exists  $\lambda \in H_{S^1}^*(X;\R)$
such that
\begin{equation} \label{eulereq}
\lambda \,e^{S^1}(N_X) = \prod_{F \subset \left( M \smallsetminus
X\right)^{S^1}}
 \negthickspace
\left ( t + \frac{[\omega|_X] }{\phi(F) - \phi(X)}
\right )^{\frac{1}{2}\dim(F) + 1}
\negthickspace \negthickspace \negthickspace
\negthickspace \negthickspace \negthickspace
 \negthickspace \negthickspace
\negthickspace \negthickspace \negthickspace
\negthickspace \negthickspace \negthickspace
.
\end{equation}
\end{enumerate}
\end{lemma}

\begin{proof}
By Lemma~\ref{uvt}, there
exists $\ut \in H^2_{S^1}(M;\R)$ so that
$$\ut|_F = [\omega|_F] + t \left( \phi(X) - \phi(F) \right)$$
for each fixed component $F$.
Hence
$$\big(\ut|_F  + t \left(\phi(F) - \phi(X)\right) \big)
^{\frac{1}{2}\dim(F) +1}=0$$ for all $F$,
and so the class
$$\prod_{F \subset \left( M \smallsetminus X\right)^{S^1}}
\big( \ut + t \left(\phi(F) - \phi(X) \right) \big)
^{\frac{1}{2}\dim(F) + 1 } $$
vanishes when restricted to any fixed component other than $X$.
Therefore, Proposition~\ref{multiple-Euler} (applied to $-\phi$) implies
that
\begin{equation*}
\prod_{F \subset \left(M \smallsetminus
X\right)^{S^1}} \big( [\omega|_X] + t (\phi(F) - \phi(X) )
\big)^{\frac{1}{2}\dim(F) + 1 } = \lambda' \, e^{S^1}(N_X)
\end{equation*}
for some $\lambda' \in H_{S^1}^*(X;\R)$.  Dividing both sides of
this equation by a suitable constant, we have (2).

Finally, if $\sum_{F \subset M^{S^1}} \left (\dim(F) + 2 \right)
= \dim(M) + 2$, then
$$\sum_{F \subset \left (M \smallsetminus X \right)^{S^1}} \left( \dim(F) + 2 \right)
= \dim(M) - \dim(X).$$ Therefore, $\lambda \in \R$, and by comparing
the coefficients of $t^{\frac{1}{2}\dim(M) - \frac{1}{2}\dim(X)}$ on
both sides of \eqref{eulereq}, we see that $\lambda \Lambda_X = 1$.
\end{proof}

\begin{remark}
More generally,
if $\sum_{F \subset M^{S^1}} \left (\dim(F) + 2 \right) = \dim(M) + 2$
and  $F$ is any fixed component, then
$$ e^{S^1}(N^-_F) = \Lambda^-_F \prod_{\phi(F') < \phi(F)}
\left ( t + \frac{[\omega|_F] }{\phi(F') - \phi(F)}
\right )^{\frac{1}{2}\dim(F') + 1} \in H_{S^1}^*(F;\R),$$ where
$e^{S^1}(N^-_F)$ is the equivariant Euler class of the
negative normal bundle of
$F$,  $\Lambda^-_F$ is the product of the weights (with multiplicity) in
the negative normal bundle of $F$, and
the product is over all fixed components $F'$ such
that $\phi(F') < \phi(F)$.
The proof for this more general case is nearly identical, except that
it uses the fact that
$$ \sum_{\phi(F') < \phi(F)} \left( \dim(F') + 2 \right) =
2\lambda_F,$$
which follows from Lemma~\ref{index'} and Remark~\ref{extend}.
\end{remark}

\begin{lemma}\label{primitive}
Let the circle act
on a compact connected symplectic manifold $(M, \omega)$
with moment map $\phi \colon M \to \R$.
Let $X$ be the minimal fixed component; assume that
$2 <  2 \lambda_F$ for all other fixed components $F$.
If $[\omega]$ is a primitive integral class, so is $[\omega|_X]$.
\end{lemma}

\begin{proof}
Since $\phi$ is a Morse-Bott function and  $2 < 2 \lambda_F$ for
all other fixed components $F$,
the natural restriction map from $H^2(M;\Z)$ to $H^2(X;\Z)$
is an isomorphism.
\end{proof}

\begin{lemma}\label{ring}
Let the circle act on a compact connected symplectic manifold $(M,
\omega)$ with moment map $\phi \colon M \to \R$. Let $X$ be the
minimal fixed component; assume that $\dim(X) < 2 \lambda_F$ for all
other fixed components $F$. If
$\lambda [\omega|_X]^j$ is an  integral class for some
$j \in \left\{0,\dots,\frac{1}{2}\dim(X)\right\}$ and $\lambda \in \R$,
then
$ \lambda \left(\phi(X) - \phi(F)\right)^j \in \Z$
for each fixed component $F \subset M^{S^1}$.
\end{lemma}

\begin{proof}
By Lemma~\ref{uvt}, there exists $\ut \in  H^2_{S^1}(M;\R)$ such
that $\ut|_F = [\omega|_F] + t(\phi(X) - \phi(F))$  for
each fixed component $F$.
Since $\phi$ is a Morse-Bott function and
$\dim(X) < 2\lambda_F$ for all fixed
components $F$ other than $X$,  the natural
restriction
map from  $H_{S^1}^{2i}(M;\Z)$ to $H^{2i}_{S^1}(X;\Z)$ is an isomorphism for all
$i \in \left\{0,\dots,\frac{1}{2}\dim(X)\right\}$.
Therefore, if $\lambda\,  \ut^j|_X = \lambda [\omega|_X]^j$ is
an integral class, then so is $\lambda \, \ut^j$.
Therefore, for any $y$ in a fixed component $F$,
$\lambda \ut^j|_y = \lambda  (\phi(X) - \phi(F))^j \, t^j$
is integral.
\end{proof}

\section{The case that there are only two fixed components}

In this section, we turn to considering the implications
of our other main restriction -- the
assumption that there are only two fixed components, $X$ and $Y$.
The key idea is  to exploit the
fact that each (nonempty regular) reduced space
is a bundle over $X$ {\em and} a bundle over  $Y$;
more specifically, it is the projectivization of the normal bundle
to $X$ and of the normal bundle to $Y$.

\begin{proposition}\label{TW}
Let the circle act  on a connected compact symplectic manifold
$(M, \omega)$  with moment map $\phi \colon M \to \R$. Let $X$ be
the maximal fixed component and fix $i \in \N$. If the action is
semifree, or if $H^*(X;\Z)$ is torsion-free, or if $i  \leq \dim(M) - \dim(X)$,
let $R = \Z$; otherwise, let $R = \R$. Given a regular value $c \in \R$
so that $M^{S^1} \cap \phi^{-1}(c, +\infty) = X$, there is an isomorphism
\begin{gather*}
\kappa_{X,c} \colon H_{S^1}^i(X; R)/  e^{S^1}(N_X)
\stackrel{\simeq}{\longrightarrow}
H_{S^1}^i(\phi^{-1}(c);R) \quad \mbox{such that}\\
\kappa_{X,c}(\at|_X) = \at|_{\phi^{-1}(c)}
 \quad \forall \  \at \in H^i_{S^1}(M;R).
\end{gather*}
Here, $e^{S^1}(N_X)$ is the equivariant Euler class
of the normal bundle to $X$.
\end{proposition}

\begin{proof}
Assume first that the action is semifree, or $H^*(X;\Z)$ is
torsion-free, or  $R = \R$.  Then this claim is a special case of
the theorem on the cohomology of reduced spaces proved in \cite{TW};
see Theorem 3 and Propositions 6.4 and 6.7.

Alternatively, as
we showed in \S\ref{background},
if any of these criteria holds {\em or} if $i \leq 2 \lambda_X$ and $R=\Z$,
the long exact sequence in equivariant cohomology for the
pair $(N_X, N_X \smallsetminus X)$ breaks into a short exact
sequence:
$$0 \to H^i_{S^1}(N_X, N_X \smallsetminus X;R) \to
H_{S^1}^i(N_X ;R) \to H^i_{S^1}(N_X \smallsetminus X;R) \to
0.$$
Since $N_X \sim X$ and $N_X\smallsetminus X\sim
\phi^{-1}(c)$, by
the Thom isomorphism theorem we can
rewrite this short exact sequence as follows:
$$
0 \to H^{i - \dim(M) + \dim(X)}_{S^1}(X;R) \to H_{S^1}^i(X;R) \to
H^i_{S^1}(\phi^{-1}(c);R) \to
0, $$
where the second arrow is multiplication by $e^{S^1}(N_X)$.
\end{proof}

If there are exactly two fixed sets, this has the following
corollary:

\begin{corollary}\label{TW'}
Let the circle act on a  compact symplectic manifold $(M,\omega)$
with moment map $\phi \colon M \to \R$.  Assume that $M$ has exactly
two fixed components, $X$ and $Y$. Fix $i \in \N$. If the action is
semifree, or if $i \leq \min \{\dim(M)- \dim(X),\dim(M) -
\dim(Y)\}$, let $R = \Z$; otherwise, let $R = \R$. There is an
isomorphism
\begin{gather*}
f \colon H_{S^1}^i(X; R)/  e^{S^1}(N_X)  \stackrel{\simeq}{\longrightarrow}
H_{S^1}^i(Y; R)/  e^{S^1}(N_Y)
\quad \mbox{such that}\\
f (\at|_X) = \at|_Y \quad \forall \ \at \in H^i_{S^1}(M;R).
\end{gather*}
Moreover,
\begin{gather*}
f\left([\omega|_X] \right) = [\omega|_Y] + t \big(\phi(X) - \phi(Y)\big) \quad \mbox{and} \\
s\, f([\omega|_X]) + (1-s) [\omega|_Y] \neq 0 \quad \forall  \ s \in
(0,1)\quad \mbox{when} \,\, \dim(M) > 2.
\end{gather*}

\end{corollary}

\begin{proof}
For simplicity, we may assume that $\phi(X) < \phi(Y)$.

By  Proposition~\ref{TW},  for any $c \in \big(\phi(X),\phi(Y) \big)$
$$ f =   \left(\kappa_{Y,c}\right)^{-1}  \circ  \kappa_{X,c} \colon
H_{S^1}^i(X; R)/  e^{S^1}(N_X) \to
H_{S^1}^i(Y; R)/  e^{S^1}(N_Y)$$
is an isomorphism
such that
$f (\at|_X) = \at|_Y$   for all  $\at \in H^i_{S^1}(M;R)$.

By Lemma~\ref{uvt},
there exists $\ut \in H_{S^1}^*(M;\R)$ such
that $\ut|_X = [\omega|_X]$ and  $\ut|_Y = [\omega|_Y] +
t \big(\phi(X) - \phi(Y) \big).$
Therefore,
$f\left([\omega|_X] \right) = [\omega|_Y] + t \big(\phi(X) - \phi(Y)\big).$

Finally, fix any $s \in (0,1)$ and let $$c = s \phi(X) + (1-s)\phi(Y) \ \in \ \big(\phi(X),\phi(Y)\big).$$
Since $c$ is a regular value,
Lemma~\ref{uvt} implies that
$\kappa_c\left(\ut - t ( \phi(X) - c) \right) = \omega_c$,
where $\kappa_c $ is the Kirwan map and $(M_c,\omega_c)$
is the symplectic reduction of $M$ at $c$.
Therefore, under the identification of  $H_{S^1}^*(\phi^{-1}(c);\R)$ and  $H^*(M_c;\R)$,
\begin{gather*}
\kappa_{Y,c} \left( s\, f\left([\omega|_X]\right) + (1-s) [\omega|_Y] \right)  \\
 =
\kappa_{X,c}
\Big(s \, [\omega|_X] + (1-s)\, \big( [\omega|_X] - t \left(\phi(X) - \phi(Y) \right) \big)\Big) \\
 =  \kappa_{X,c}\big([\omega|_X] -
 t (\phi(X) - c ) \big) \\
 =  \big(\ut -  t (\phi(X) - c ) \big)\big|_{\phi^{-1}(c)} \\
 =  \kappa_{c}\big(\ut -  t (\phi(X) - c ) \big) \\
 = \omega_c.
\end{gather*}
Since $\omega_c \neq 0$ when $\dim(M_c) > 0$, the final claim
follows immediately.
\end{proof}

It is particularly easy to analyze the case
where one of the two fixed components has codimension two.

\begin{remark}\label{two}
Consider an effective  Hamiltonian circle action on a compact symplectic
manifold $(M,\omega)$. Assume that $M^{S^1}$ has exactly two
components, $X$ and $Y$, and that $Y$ has codimension two.
Then the fixed set
data near $Y$ is determined by the data near $X$.
More precisely,  there is a natural isomorphism
$$H^*(Y;\Z) = H^*_{S^1}(X;\Z)/e^{S^1}(N_X);$$
under this identification, $$e(N_Y) = t \quad
\mbox{and} \quad c(Y) = c(X) c^{S^1}(N_X).$$
To see this note that, since  $\rank_\C(N_Y) = 1$, the action  must be semifree.
Moreover, since $e^{S^1}(N_Y) = -t + e(N_Y)$,
the inclusion $H^*(Y;\Z) \to H_{S^1}^*(Y;\Z)$
induces an isomorphism from $H^*(Y;\Z)$ to $H_{S^1}^*(Y;\Z)/ e^{S^1}(N_Y)$.
Finally,
$$c^{S^1}(M)\big|_X = c(X)c^{S^1}(N_X)
\quad \mbox{and} \quad
c^{S^1}(M)\big|_Y =  c(Y)\left(1+ e^{S^1}(N_Y)\right).$$
Therefore, the claims  follow immediately from Corollary~\ref{TW'}.

\end{remark}
Additionally, when the two fixed components have minimal dimension,
the Chern class of each component is determined by the Chern class
of its normal bundle and the weights of the isotropy action on the
other component.

\begin{lemma}\label{qchern}
Let the circle act on a compact symplectic manifold $(M, \omega)$
with moment map $\phi \colon M \to \R$. Assume that $M^{S^1}$ has
exactly two components, $X$ and $Y$, where $\dim(X)+\dim(Y) = \dim(M)-2$.
Under the natural isomorphism
$$H^*(X) \simeq H_{S^1}^*(X)/\left([\omega|_X]  + t \big(\phi(Y) - \phi(X) \big)\right)$$
the total Chern class of $X$ is $$c(X)
= \frac{ \prod_{\lambda} \left( 1 + \lambda t \right)} {
c^{S^1}(N_X) }, $$ where the product is over the weights
(counted with multiplicity)
$\lambda$ in $N_Y$.
Here, $N_X$ and $N_Y$ are the normal bundles to $X$ and $Y$, respectively.
\end{lemma}

\begin{proof}
By Corollary~\ref{TW'}, there is an isomorphism
\begin{gather*}
f \colon H^*_{S^1}(X;\R)/ e^{S^1}(N_X)  \to
H^*_{S^1}(Y;\R)/ e^{S^1}(N_Y)  \quad \mbox{such that} \\
f(t) = t, \quad
f \Big (c^{S^1}(M)\big|_X \Big) = c^{S^1}(M)\big|_Y,
\quad \mbox{and} \quad
f(u) = v - mt,
\end{gather*}
where $u = [\omega|_X]$,
$v = [\omega|_Y]$, and
$m = \phi(Y) - \phi(X)$.

Fix a point $y \in Y$.
Since $f(u + mt)|_y = 0$,
the composition of $f$
and the restriction map $$\iota_y^* \colon H_{S^1}^*(Y;\R)/e^{S^1}(N_Y)
\to H_{S^1}^*(\{y\};\R)/t^{\frac{1}{2}\dim(M) - \frac{1}{2}\dim(Y)}$$
induces a  map
\begin{gather*}
g \colon H_{S^1}^*(X;\R)/\big( u + mt,  e^{S^1}(N_X)\big) \to
H_{S^1}^*(\{y\};\R)/t^{\frac{1}{2}\dim(M) - \frac{1}{2}\dim(Y)} \quad \mbox{so that} \\
g(u) = -mt \quad \mbox{and} \quad g \big(c^{S^1}(M)|_X \big) =
c^{S^1}(M)\big|_y =  \prod_\lambda (1 + \lambda t),
\end{gather*}
where again the product is over all the weights $\lambda$ in $N_Y$.
 Moreover, since $\dim(X) + \dim(Y) = \dim(M) - 2$,
Lemma~\ref{euler} implies that
$$ e^{S^1}(N_X)=\Lambda_X\left(t+\frac{u}{m}\right)^{\frac{1}{2}\dim(Y)+1};$$
in particular,  $e^{S^1}(N_X)$ is a multiple of $u + mt$.
Therefore,
$$H_{S^1}^*(X;\R)/ \big(u + mt, e^{S^1}(N_X) \big)
= H_{S^1}^*(X;\R)/\left( u + mt \right) .$$
Finally, Proposition~\ref{implies}
implies that
$$
H^*(X;\R) = \R[u]/u^{\frac{1}{2}\dim(X) + 1}
.$$
Therefore, $g$ is an isomorphism.
Since $c^{S^1}(M)|_X = c(X)c^{S^1}(N_X)$,
the claim follows immediately.
\end{proof}

\section{Proof of  Theorem~\ref{B} for semifree actions}

In this section, we prove Theorem~\ref{B} in the case when the
circle action is
semifree.

\begin{proposition}\label{case1}
Let the circle  act on a compact symplectic manifold
$(M, \omega)$  with moment map $\phi \colon M \to \R$. Assume that
$M^{S^1}$ has exactly two components, $X$ and $Y$,
and that $\dim(X)+\dim(Y)=\dim(M) - 2$.
Also assume that the action is semifree.
 Then
\begin{align*}
H^*(X;\Z) = \Z[u]/u^{\frac{1}{2} \dim (X) + 1}
& \quad \mbox{and} \quad
H^*(Y;\Z) = \Z[v]/v^{\frac{1}{2} \dim (Y) + 1} ;\\
c(X) = (1 + u)^{\frac{1}{2} \dim (X) + 1} \
& \quad \mbox{and} \quad
c(Y)=(1+v)^{ \frac{1}{2} \dim (Y) +1}; \\
c({N_X})=(1+u)^{\frac{1}{2}\dim (Y)+1}
& \quad \mbox{and} \quad
c({N_Y}) =(1+v)^{\frac{1}{2} \dim(X) +1};
\end{align*}
where $N_X$ and $N_Y$ denote the normal bundles to $X$ and $Y$,
respectively.
Moreover, if $\phi(Y) > \phi(X)$ and $[\omega]$ is a primitive integral class, then
$\phi(Y) - \phi(X) = 1$.
\end{proposition}

\begin{proof}
Clearly, the proposition holds if  $\dim(X)=\dim(Y)=0$.
Without loss of generality, we assume that $\phi(X) < \phi(Y)$ and that $\dim(X) > 0$.
By assumption, there exist natural numbers $i > 0$ and $j$ such that
\begin{gather}\label{semidim}
\dim(X) = 2i,  \ \
\dim(Y) = 2j,
\ \  \mbox{and} \ \
\dim(M) = 2i + 2j + 2;  \quad \mbox{hence}\\
\label{semirank}
\rank_\C (N_X) = j+1 \quad \mbox{and} \quad \rank_\C(N_Y) = i + 1.
\end{gather}
By Proposition~\ref{implies} and Lemma~\ref{euler},
\begin{gather} \label{semiring}
H^*(X;\R) = \R[u]/u^{i+1}, \quad \mbox{where} \ u = [\omega|_X], \quad \mbox{and} \\
\label{semieuler} e^{S^1}(N_X)=\left( t+ \frac{u}{m}\right)^{j+1}, \quad \mbox{where} \ m = \phi(Y) - \phi(X).
\end{gather}
Since the action is semifree,  \eqref{semieuler} and Lemma~\ref{Chern}
imply that the total equivariant Chern class of $N_X$ is
\begin{equation} \label{chernsemin}
c^{S^1}(N_X)=\left(1+t+\frac{u}{m}\right)^{j+1}.
\end{equation}
Similarly,  \eqref{semirank}, \eqref{chernsemin}, and Lemma~\ref{qchern} imply
that the total Chern class of $X$ is
\begin{equation}\label{chernsemix}
c(X) =\left(1+\frac{u}{m}\right)^{i+1}= 1 + (i+1) \frac{u}{m} +
\dots + (i +1)\frac{u^i}{m^i}.
\end{equation}
By \eqref{semiring} and  \eqref{chernsemix}, the Euler characteristic of $X$ is
\begin{equation}\label{semie}
i+1 =  \sum_k (-1)^k \dim \big(H^k(X;\R)\big)  =  \int_X c_i(X) = (i +1)
\int_X \frac{u^i}{m^i}.
\end{equation}
Therefore, $\frac{u^i}{m^i}  \in H^{2i}(X;\R)$ is a primitive
integral class. By multiplying $[\omega]$ by a constant, we may
assume that $[\omega]$ is also  a primitive integral
class. Hence, $u = [\omega|_X] \in H^2(X;\R)$ is a primitive integral class by
Lemma~\ref{primitive}. By Poincar\'e duality, these two facts  imply
that $\frac{u^{i-1}}{m^i}  \in H^{2i-2}(X, \R)$ is an integral
class. By Lemma~\ref{ring}, this implies that $m^i$ divides
$m^{i-1}$, that is,
\begin{equation}\label{msemi}
 m=1.
\end{equation}
By Proposition~\ref{implies}, this implies that
\begin{equation}\label{semi}
H^*(X;\Z) = \Z[u]/u^{i + 1}.
\end{equation}
Since nearly identical arguments can be applied to $Y$,
the claim now follows immediately from
\eqref{semidim},
\eqref{chernsemin},
\eqref{chernsemix},
\eqref{msemi},   and
\eqref{semi}.

\end{proof}

\section{Isotropy submanifolds}

Let the circle act effectively on a compact symplectic manifold $(M, \omega)$.
If the action is not semifree, then the assumption that there are only
two fixed components induces strong restrictions on $M$ itself and on its isotropy
submanifolds, especially if the fixed components have relatively simple cohomology.
Here, an {\bf isotropy submanifold} is a symplectic submanifold $M^{\Z_k} \subsetneq M$
which is not fixed by the $S^1$ action, but is fixed by the $\Z_k$
action for some $k > 1$.

We begin with some results which do not depend on the
cohomology of the fixed components.

\begin{lemma}\label{i<j}
Let the circle act effectively on a connected compact symplectic manifold
$(M,\omega)$ with moment map $\phi \colon M \to \R$.  Assume that
$M$ has exactly two fixed components, $X$ and $Y$.
\begin{itemize}
\item
If the action is not semifree, then
$\dim(X) = \dim(Y).$
\item Given an isotropy submanifold $Q \subsetneq M$,
there exists a cohomology  class $\at \in H_{S^1}^{\dim(Q) - \dim(X)}(M;\Z)$ so that
$$\at|_X =  e^{S^1}(N_X^Q)
\quad  \mbox{and} \quad
\at|_Y =  \pm e^{S^1}(N_Y^Q) ,$$
where $N_X^Q$ and $N_Y^Q$ are the normal bundles of $X$ and $Y$ in $Q$.
\end{itemize}
\end{lemma}

\begin{proof}
Assume that the action is not semifree, and fix any isotropy
submanifold $Q \subsetneq M$. Consider a cohomology class $\mu \in
H^{\dim(Q) - \dim(X)}_{S^1}(M;\R)$. By applying Corollary~\ref{TW'}
to $\mu|_Q \in H_{S^1}^{\dim(Q) - \dim(X)} (Q;\R)$,
we see that $\mu|_X$ is a multiple of $e^{S^1}(N_X^Q)
\in H_{S^1}^*(X;\R)$ exactly if $\mu|_Y$ is a multiple of
$e^{S^1}(N_Y^Q) \in H_{S^1}^*(Y;\R)$.

Assume first that $\dim(X) > \dim(Y)$.
Since \eqref{short} is exact for $R = \R$, there
exists  $\at \in H_{S^1}^{\dim(Q) - \dim(X)}(M;\R)$ such
that $\at|_X = e^{S^1}(N_X^Q)$.
By the first paragraph, $\at|_Y = a e^{S^1}(N_Y^Q)$
for some $a \in H^{\dim(Y) - \dim(X)}_{S^1}(Y;\R)$.
Since $\dim(Y) - \dim(X) < 0$, this implies that $\at|_Y = 0$.
Now we apply Corollary~\ref{TW'} to $\at$ on $M$;  since $\at|_Y = 0$,
 $\at|_X$ is a multiple
of $e^{S^1}(N_X)$, where $N_X$  is the normal bundle of $X$  in $M$.
Since $\deg\left(\at\right) <\deg\big(e^{S^1}(N_X) \big)$,  this implies that $\at|_X = 0$,
which gives a contradiction.
Therefore, $$\dim(X) = \dim(Y).$$

Since $\dim(Q) - \dim(X) = \dim(Q) - \dim(Y)$, Corollary~\ref{TW'}
(applied to $\mu|_Q$)
implies that for any cohomology class $\mu \in H^{\dim(Q) -
\dim(X)}_{S^1}(M;\Z)$, $\mu|_X$ is an {\em integer} multiple of
$e^{S^1}(N_X^Q) \in H_{S^1}^*(X;\Z)$ exactly if $\mu|_Y$ is an {\em
integer} multiple of $e^{S^1}(N_Y^Q) \in H_{S^1}^*(Y;\Z)$. Moreover,
since
$\dim(Q) - \dim(X) < \dim(M) - \dim(Y) = 2 \lambda_Y$,
\eqref{short} is exact if we take $R = \Z$, $j = \dim(Q) - \dim(X)$,  and $F = Y$.
Thus, there exists an {\em integral} class  $\at \in
H_{S^1}^{\dim(Q) - \dim(X)}(M;\Z)$ such that $\at|_X =
e^{S^1}(N_X^Q)$.
Similarly, there exists
$\bt \in H_{S^1}^{\dim(Q) - \dim(Y)} (M;\Z)$ such that $\bt|_Y =
e^{S^1}(N_Y^Q)$.
By the argument above, $\at|_Y = a e^{S^1}(N_Y^Q)$
for some $a \in \Z$
and $\bt|_X = b
e^{S^1}(N_X^Q)$ for some $b \in \Z$.
Then $(\bt - b \at)|_X = 0$.
Hence  the same argument as the last paragraph and the fact that
$$\deg\big(\bt - b \at \big) =  \dim(Q) - \dim(Y) < \dim(M) - \dim(Y) = \deg\big(e^{S^1}(N_Y)  \big),$$ where $N_Y$ is the normal bundle of $Y$ in $M$,
yield that $(\bt - b\at)|_Y = 0$. On the other hand, by a direct computation,   $(\bt -
b \at)|_Y = (1 - ab) e^{S^1}(N_Y^Q)$. This is only possible if
$ab=1$,  which implies that
$a = b = \pm 1$.
\end{proof}

\begin{corollary}\label{same}
Let the circle act effectively on a compact symplectic manifold
$(M,\omega)$ with moment map $\phi \colon M \to \R$.  Assume that
$M$ has exactly two fixed components, $X$ and $Y$,
and that the action is not semifree.
Then $$\Xi_X = - \Xi_Y,$$
where $\Xi_X$ and $\Xi_Y$ denote the multisets of
weights (counted with multiplicity) for the isotropy action on $N_X$
and $N_Y$, respectively.
\end{corollary}

\begin{proof}
Consider any $k > 1$.
Since $M^{S^1}$ has only two components, if there exists any points with stabilizer
$\Z_k$, then the
isotropy submanifold  $M^{\Z_k}$ is connected and  contains $X$ and $Y$.
Moreover, since the action is not semifree, $\dim(X) = \dim(Y)$ by Lemma~\ref{i<j}.
Therefore,  $k$ divides exactly the same number of
weights in $\Xi_X$ and $\Xi_Y$.
\end{proof}

\begin{lemma}\label{I}
Let $A$ be a set of relatively prime natural numbers $a_1< \dots < a_N$.
Assume that for each $i$ and $k$ in
$\{1,\dots,N\}$, there exists $j \in \{1,\dots,N\}$ such that $a_i
+ a_j = 0 \mod a_k$. Then  $a_i = i$ for all $i$.
\end{lemma}

\begin{proof}

 The claim is obvious if $N = 1$.  Assume that the
claim holds for $N-1$.

Consider any $i \in \{1,\dots,N-1\}$. By assumption, there  exists
$j \in \{1,\dots,N\}$ such that $a_i + a_j = 0 \mod a_N$. Since $a_i
< a_N$ and $a_j \leq a_N$, this implies that $a_i + a_j = a_N$.
Since $a_1 < \dots < a_{N-1}$, this immediately implies that $$a_i +
a_{N-i} = a_N \quad \forall \ i \in \{1,\dots,N-1\}.$$

Let $A' = \{a_1,\dots,a_{N-1}\}$. Since the elements of $A$  are relatively prime,
the  equation above immediately implies that the
elements of $A'$ are relatively prime.
Moreover, fix $i$ and $k$  in $\{1,\dots,N-1\}$. By assumption, there exists $j \in
\{1,\dots,N\}$ such that $a_i + a_j = 0 \mod a_k$.  Moreover, if $j
= N$, then since $a_k + a_{N-k}= a_N$ this implies that $a_i + a_k
+ a_{N-k} = 0 \mod a_k$, and hence $a_i + a_{N-k} = 0 \mod a_k$. By
the inductive hypothesis, this implies that $A' = \{1,\dots,N-1\}$.
The result follows immediately.
\end{proof}

\begin{lemma}\label{modulo}
Let the circle act on a compact symplectic manifold $(M, \omega)$.
Let $p$ and $q$ be fixed points which lie on the same component $N$
of $M^{\mathbb Z_k}$ for some $k>1$. Then the weights of the action
at $p$ and at $q$ are equal modulo $k$.
\end{lemma}

For a proof of this lemma, see Lemma 2.6 in \cite{T}.

\begin{proposition}\label{distinctweights}
Let the circle act effectively on a compact symplectic manifold
$(M,\omega)$ with moment map $\phi \colon M \to \R$.  Assume that
$M^{S^1}$ has exactly two components, $X$ and $Y$.
Then there exists $N \in \N$ so that
the set of distinct
weights for the isotropy action on $N_X$ is
$\{1,\dots,N\}$.
\end{proposition}

\begin{proof}
Let $A = \{a_1,\dots,a_N\} \subset \N$ be the set of distinct
weights for the isotropy action on $N_X$.
By Corollary~\ref{same}
the set of distinct weights for the isotropy action on $N_Y$ is
$\{-a_1,\dots,-a_N\}$. Moreover,
by Lemma~\ref{modulo},
for each $i$ and $k$ in $\{1,\dots,N\}$,
there exists $j \in \{1,\dots,N\}$ such
that $a_i = -a_j \mod a_k$. Finally, since the action is effective,
$a_1,\dots,a_N$ are relatively prime. Therefore,  Lemma~\ref{I}
implies that  $A=\{1, 2, \cdots, N \}$ for some $N\in\N$.
\end{proof}

The remaining results depend on the cohomology of
the fixed components.

\begin{lemma}\label{r=1}
Let the circle act effectively on a compact symplectic manifold $(M, \omega)$
with moment map $\phi \colon M \to \R$.
Assume that
$M^{S^1}$ has exactly two components, $X$ and $Y$.
Assume that $b_2(X) = 1$,
and let  $Q \subsetneq M$ be an isotropy submanifold such that
$\dim(Q) - \dim(Y)  = 2.$
Then $$c_1(N_X^Q)= 0,$$
where $N_X^Q$ denotes the normal bundle to $X$ in $Q$.
\end{lemma}

\begin{proof}
By Lemma~\ref{i<j}, $\dim(X) = \dim(Y)$.
Since $\dim(M) > \dim(Q)$, the fact that $\dim(Q) - \dim(Y) = 2$
implies that $\dim(M) - \dim(X) = \dim(M) - \dim(Y)  > 2$.
Hence, $ H^2(M;\R) =
H^2(X;\R) =H^2(Y; \R)= \R$. In particular, after possibly multiplying
$[\omega]$ by a constant,  we may assume that  $[\omega]$ is a
primitive integral class. The induced $S^1/\Z_q$ action on $Q =
M^{\Z_q}$ is semifree, and the moment map for this action is $\phi'
= \frac{\phi}{q}$. Let $u =[\omega|_X]$, $v =[\omega|_Y]$, and $m =
\phi'(Y) - \phi'(X)$.

Since $\dim(Q) - \dim(Y) = 2$, $e^{S^1}(N_Y^Q) = -t + e(N_Y^Q)$
and so $H^*(Y;\Z) \simeq  H^*_{S^1}(Y;\Z)/e^{S^1}(N_Y^Q)$
(see Remark~\ref{two});
similarly, $e^{S^1}(N_X^Q) = t + e(N_X^Q)$ and
$H^*(X;\Z) \simeq  H^*_{S^1}(X;\Z)/e^{S^1}(N_X^Q)$.
Therefore,
by Corollary~\ref{TW'} (applied on $Q$), there exists an
isomorphism $f \colon H^*(Y;\Z) \to H^*(X;\Z)$ so that $f(v) = u
-m\, e(N^Q_X)$
and so that $s f(v) + (1-s) u \neq 0$ for all $s \in (0,1)$.
On the one hand, by Lemma~\ref{primitive}, both $u$ and $v$ are
primitive integral classes.
Since $f$ is an isomorphism, $f(v)$ is also
primitive. Since $H^2(X;\R)  = \R$, this implies that $f(v)=\pm u$.
On the other hand, since $H^2(X;\R) = \R$, the fact that $s f(v) +
(1-s)u \neq 0$ for all $s \in (0,1)$ implies that $f(v)$ is a {\em
positive} multiple of $u$. Together, these two claims imply that
$f(v) = u$. Since $f(v) = u - m \, e(N^Q_X)$ and $m \neq 0$, this
implies that $c_1(N_X^Q) = e(N_X^Q) = 0$.
\end{proof}

\begin{lemma}\label{1stChern}
Let the circle act on a  compact  symplectic manifold $(M,
\omega)$  with moment map $\phi \colon M \to \R$. Assume that
$M^{S^1}$ has exactly two components, $X$ and $Y$.
Assume that  $b_2(X) = 1$,  and
let $Q \subsetneq M$ be an isotropy submanifold such that $\dim(Q) - \dim(Y)
> 2$.
Then
$$ c_1(N_Q|_X)=   2 \, \Gamma_Q \frac{u}{m},$$
where  $N_Q$ is the normal bundle of $Q$ in $M$,
$\Gamma_Q$ is the sum  of the weights (counted with
multiplicities) of the isotropy action on $N_Q|_X$,
$m=\phi(Y)-\phi(X)$, and $u = [\omega|_X]$.
\end{lemma}

\begin{proof}
By Corollary~\ref{same}, the sum of the weights
(counted with multiplicity) of the $S^1$ action on $N_Q\big|_Y$ is $-\Gamma_Q$.
Hence,
\begin{gather*}
c_1^{S^1}(N_Q)\big|_x = \Gamma_Q\, t \quad \forall \ x \in X ,\quad
\mbox{and}\quad
 c_1^{S^1}(N_Q)\big|_y = - \Gamma_Q\, t \quad \forall \ y \in Y.
\end{gather*}

By Lemma~\ref{uvt}, there exists $\ut \in H^2_{S^1}(Q;\R)$
such that $\ut|_X = u$ and $\ut|_Y = v - mt$, where
$v = [\omega|_Y]$.
Since $H^2(X;\R) = \R$ and $\dim(Q) - \dim(Y) > 2$,
there exists $a$ and $b$ in $\R$ such that
$$c_1^{S^1}(N_Q) = a \ut + b (\ut + mt) .$$
Therefore,
\begin{gather*}
c_1^{S^1}(N_Q)\big|_x  =  b m t \quad \forall \ x \in X , \quad
\mbox{and}\quad c_1^{S^1}(N_Q)\big|_y  =  - a m t \quad  \forall \ y
\in Y .
\end{gather*}
The claim follows immediately.

\end{proof}

\begin{lemma}\label{alpha}
Let the circle act effectively on a  compact  symplectic manifold $(M,
\omega)$  with moment map $\phi \colon M \to \R$. Assume that $M^{S^1}$ has
exactly two components, $X$ and $Y$.
Assume that $b_{2i}(X) = 1$
for all $i \in \left\{0,\dots,\frac{1}{2}\dim(X)\right\}$.
Finally, assume that the action is not semifree, and split
$N_X = \bigoplus_k V_k$, where
$N_X$ is the normal bundle to $X$ in $M$
and  $V_k \subset N_X$ is the subbundle on which $S^1$
acts with weight $k$.
Let $V \subset N_X$
be the direct sum of some subset of the $V_k$'s.
If $\rank_\C V > 1$, then
$$c_1(V)=  \nu \, \Gamma_V \frac{u}{m},
\quad \mbox{where}\ 0 <  \nu  < 2. $$
Here,
$\Gamma_V$ is the sum of the weights (counted with
multiplicities) of the $S^1$ action on $V$,  $m = \phi(Y) -
\phi(X)$ and $u = [\omega|_X]$.
\end{lemma}

\begin{proof}
Let $Q = M^{\Z_q} \subsetneq M$ be an isotropy submanifold.
By Lemma~\ref{i<j}, $\dim(X) = \dim(Y)$, and  there exists
$\at \in H^{2r}_{S^1}(M;\R)$ so that
\begin{equation}\label{<1}
\at|_X = e^{S^1}(N_X^Q) \quad \mbox{and} \quad \at|_Y = \pm e^{S^1}(N_Y^Q).
\end{equation}
Here,
$N_X^Q$ and $N_Y^Q$
denote the normal bundles of $X$ and $Y$, respectively, in $Q$,
and $\dim(Q) - \dim(X) = \dim(Q) - \dim(Y) = 2r$.

Let $\Lambda_X^Q$ denote the product of the weights (counted with
multiplicity) of the isotropy action on $N_X^Q$. By
Corollary~\ref{same},
the product of the weights
of the isotropy action on $N_Y^Q$ is $(-1)^r\Lambda_X^Q$.
Hence,
\begin{gather}\label{<2}
e^{S^1}\big(N_X^Q\big)\big|_x =  \Lambda_X^Q t^r \ \  \forall \ x \in X,
\ \  \mbox{and}\ \
e^{S^1}\big(N_Y^Q\big)\big|_y  = (-1)^r \Lambda_X^Q t^r \ \  \forall \ y \in Y.
\end{gather}

By Lemma \ref{uvt}, there exists $\ut \in H^2_{S^1}(M;\R)$ such that
$\ut|_X = u$ and $\ut|_Y = v - mt$, where $v = [\omega|_Y]$.
Since $X$ is symplectic and
$b_{2i}(X) = 1$ for
all $i \in \left\{0,\dots,\frac{1}{2}\dim(X) \right\}$,
$$H^{\even}(X;\R) = \R[u]/u^{\frac{1}{2}\dim(X) + 1}.$$
Hence, since $\dim(M) - \dim(Y) >
\dim(Q) - \dim(Y) = 2r$,
we can write
\begin{gather}\label{<3}
\at = \sum a_i \left(
\frac{\ut}{m}\right)^{i} \left(\frac{\ut}{m} + t \right)^{r-i},
\quad \mbox{where} \\
\label{<4}
\at|_x =  a_0 t^r \ \  \forall \ x \in X,
 \quad \mbox{and}\quad
\at|_y =  a_r (-t)^r \ \  \forall \ y \in Y.
\end{gather}
Combining equations \eqref{<1}, \eqref{<2},
and \eqref{<4}, we see that $a_0 = \pm a_r$.
Therefore, \eqref{<3} implies that
\begin{gather}
\label{NXQ}
e^{S^1}(N_X^Q) =  \sum a_i \left(
\frac{u}{m}\right)^i \left(\frac{u}{m} + t \right)^{r-i},
\quad \mbox{where }
a_0 = \pm a_r.
\end{gather}

On the other hand,
by Lemma~\ref{euler},
$\left(\frac{u}{m} + t \right)^{\frac{1}{2}\dim(X) +1}$
is a multiple of $e^{S^1}(N_X)$.
Since $N_X = \bigoplus_k V_k$,  and hence
$$e^{S^1}(N_X) = \prod e^{S^1}(V_k)  \quad \in
H_{S^1}^\even(X;\R) \simeq \R[u,t]/ u^{\frac{1}{2}\dim(X) + 1},$$ this
implies that the $e^{S^1}(V_k)$'s can be identified
with polynomials in $\C[u,t]$
whose product divides
$\left(  \frac{u}{m} +t \right)^{\frac{1}{2}\dim(X)+1}  +
\left( \lambda  \frac{u}{m}\right)^{\frac{1}{2}\dim(X) +1}$
for some $\lambda  \in \C$.
Write $$e^{S^1}(V_k) =
  k^{r_k} \sum \alpha_{k,i} \left( \frac{u}{m} \right)^i
\left( \frac{u}{m} + t \right)^{r_k - i},$$
where $r_k = \rank_\C V_k$; note that $\alpha_{k,0} = 1$.
Since
$$\left( \frac{u}{m} + t  \right)^{\frac{1}{2}\dim(X)+1}
+  \left( \lambda  \frac{u}{m} \right)^{\frac{1}{2}\dim(X) +1} =
\prod_{i=0}^{\frac{1}{2}\dim(X)}
\left ( \frac{u}{m} + t +   e^{\frac{i 4 \pi \sqrt{-1}}{\dim (X) + 2}} \,
\lambda \frac{u}{m}
\right),$$
this implies that for all $k$,
$$\left| {\alpha_{k,r_k}} \right| = \left| \lambda \right|^{r_k}
\quad \mbox{and} \quad \left| {\alpha_{k,1}} \right|
\leq  r_k |\lambda|.$$
Moreover,  if $r_k > 1$ then $|\alpha_{k,1}| < r_k |\lambda|$,
while if $r_k = r_{k'} = 1$, then
$\alpha_{k,1} \neq \alpha_{k',1}$ unless $k  = k'$.
Since $N_X^Q = \bigoplus_n V_{nq}$,
the  fact that $a_0 = \pm a_r$
in \eqref{NXQ}
implies that $|\lambda| = 1$.
Therefore, (since $e^{S^1}(V_k)$ is real)
$$ e^{S^1}(V_k)= (kt)^{r_k} + \nu_k  k r_k \frac{u}{m} (k t)^{r_k-1}  +
\lot,$$ where
$0 < \nu_k < 2$   for all $k$
except possibly:
\begin{itemize}
\item at most one $k$ such that $r_k = 1$ and
$\nu_k = 0$; and
\item
at most one $k$ such that $r_k = 1$ and $\nu_k = 2$.
\end{itemize}
By Lemma~\ref{Chern},
$$c_1(V_k )=\nu_k k r_k \frac{u}{m}.$$
The claim follows immediately.
\end{proof}

\begin{proposition}\label{case3}
Let the circle act effectively on a  compact  symplectic manifold $(M,
\omega)$  with moment map $\phi \colon M \to \R$.
Assume that $M^{S^1}$ has exactly
two components, $X$ and $Y$, and that
$b_{2i}(X) = 1$ for all $i \in \left\{0,\dots,\frac{1}{2}\dim(X)\right\}$.
\begin{itemize}
\item
No point has stabilizer $\Z_k$ for any $k > 2$.
\item
If the action is not semifree, then
$$
\dim(M^{\Z_2}) - \dim(Y) = 2
\quad \mbox{or} \quad \dim(M) - \dim(M^{\Z_2}) = 2 \quad \mbox {(or both)}.$$
\end{itemize}
\end{proposition}

\begin{proof}
To begin, let  $Q \subset M$
be an isotropy submanifold such that $\dim(Q) - \dim(Y) > 2$
and $\dim(M) - \dim(Q) > 2$.
Let $N_Q$ be the normal bundle of $Q$ in $M$,
$\Gamma_Q$ be the sum of the weights (counted with multiplicity)
of the isotropy action on  $N_Q|_X$,
$m = \phi(Y) - \phi(X)$,
and $u = [\omega|_X]$.
By Lemma~\ref{alpha},
$$c_1(N_Q|_X)=  \nu \,\Gamma_Q\frac{u}{m},
\quad \mbox{where}\ \nu < 2.$$
On the other hand,  the fact that $\dim(M) - \dim(Q) > 2$ implies that $\dim(M) - \dim(Y) > 2$.
Since $\phi$ is a perfect Morse-Bott function,  $\dim(X) > 0$, and  so
$b_2(X) = 1$ by assumption.
Hence, by
Lemma~\ref{1stChern},
$$c_1(N_Q|_X)=  2\, \Gamma_Q \frac{u}{m}.$$
This gives a contradiction.
Therefore, for any isotropy submanifold $Q \subsetneq M$,
\begin{equation}\label{c31}
\dim(Q) - \dim(Y) = 2
\quad \mbox{or} \quad \dim(M) - \dim(Q) = 2 \quad \mbox {(or both)}.
\end{equation}

Let $N_X$ be the normal bundle to $X$.
By Proposition~\ref{distinctweights}, there exists $N \in \N$ so that
the set of distinct weights for the isotropy action on $N_X$ is
$\{1,2,\dots,N\}$.
Split $N_X= \sum_{k=1}^N
V_k$, where $V_k$ is the subbundle of $N_X$ on which $S^1$ acts with
weight $k$.

Assume that $N > 2$.
Then it is easy to check
that $\dim(M) - \dim(M^{\Z_k}) > 2$ for all $k \in \{2,\dots,N\}$.
By \eqref{c31},  this implies that $\dim(M^{\Z_k}) - \dim(Y) = 2$ for all such $k$.
Therefore, by Lemma~\ref{r=1},
$c_1(V_{N-1})=0 $ and $c_1(V_{N})=0$,
and so $c_1(V_{N-1} \oplus V_{N}) = 0$.
This contradicts
Lemma~\ref{alpha},
which implies that $c_1(V_N \oplus V_{N-1}) \neq 0$.
\end{proof}

\section{Proof  of Theorem~\ref{B} for actions which are not
semifree}

In this section, we prove Theorem~\ref{B} in the case that the
circle action is not semifree.

\begin{proposition}\label{case2}
Let the circle act effectively on a   compact  symplectic manifold $(M,
\omega)$  with moment map $\phi \colon M \to \R$.
Assume that $M^{S^1}$ has exactly
two components, $X$ and $Y$, and that $\dim(X)+\dim(Y)=\dim(M) - 2$.
Also assume that the  action is not semifree.
Then
\begin{gather*}
H^*(X;\Z) = \Z[u]/u^{i  + 1} \ \  \mbox{and} \ \  c(X) = (1+u)^{i+1};  \\
H^*(Y;\Z) = \Z[v]/v^{i + 1} \ \  \mbox{and} \ \  c(Y) = (1+v)^{i+1};   \\
\mbox{where} \quad \dim(X) =  \dim(Y) = 2i \geq 2.
\end{gather*}
Moreover, no point has stabilizer $\Z_k$ for any  $k > 2$;
$\dim(M^{\Z_2}) = \dim(M) - 2$;
\begin{align*}
c\big(N_{M^{\Z_2}}\big)\big|_X = 1+ u &\quad \mbox{and} \quad
c\big(N_{M^{\Z_2}}\big)\big|_Y = 1 + v; \\
c\big(N_X^{M^{\Z_2}}\big) = \frac {(1+u)^{i + 1}} {1 + 2u}
&\quad \mbox{and} \quad c\big(N_Y^{M^{\Z_2}}\big) = \frac
{(1+v)^{i + 1}} {1 + 2v},
\end{align*}
where
$N_{M^{\Z_2}}$ denotes the
normal bundle of $M^{\Z_2}$ in $M$, and
$N_X^{M^{\Z_2}}$ and $N_Y^{M^{\Z_2}}$ denote the normal bundles of $X$ and $Y$,
respectively, in $M^{\Z_2}$.
\end{proposition}

\begin{proof}
This claim follows from
Lemmas~\ref{isotropy-Q}, \ref{r=s=1}, \ref{normal-bundle}, and
 ~\ref{torsion-Grass}.
\end{proof}

To begin, note that
the following lemma is an immediate
consequence of Propositions~\ref{implies} and \ref{case3}.

\begin{lemma}\label{isotropy-Q}
If the assumptions of Proposition~\ref{case2} hold, then
no point has stabilizer $\Z_k$ for any $k > 2$.
\end{lemma}

\begin{lemma}\label{r=s=1}
If the assumptions of Proposition~\ref{case2} hold and,
additionally, $\dim(M^{\Z_2})- \dim(Y) = 2$, then
\begin{align*}
\dim(M) = 6 &\quad \mbox{and} \quad \dim(X) = \dim(Y) = 2; \\
H^*(X;\Z) = \Z[u]/u^2
&\quad \mbox{and} \quad
H^*(Y;\Z) = \Z[v]/v^2 ;\\
c(X)=1+ 2u &\quad \mbox{and} \quad c(Y)=1+2v;\\
c\big(N_{M^{\Z_2}}\big)\big|_X = 1+ u &\quad \mbox{and} \quad
c\big(N_{M^{\Z_2}}\big)\big|_Y = 1 + v; \\
c\big(N_X^{M^{\Z_2}}\big) = 1
&\quad \mbox{and} \quad
c\big(N_Y^{M^{\Z_2}}\big) = 1.
\end{align*}
Moreover, if $\phi(Y) > \phi(X)$ and $[\omega]$ is a primitive integral class, then
$\phi(Y) - \phi(X) = 2$.
\end{lemma}

\begin{proof} Without loss of generality, we may assume that $\phi(X)<\phi(Y)$.
By Lemma~\ref{i<j},
since the action is not semifree,
$\dim(X) = \dim(Y)$;
hence there
exists  $i\in\N$ such that
\begin{gather}\label{dims}
\dim(X) = \dim(Y) = 2i  \quad \mbox{and}\ \dim(M) = 4i + 2.
\end{gather}
By assumption,
\begin{gather} \label{ndims}
\dim\big(M^{\Z_2}\big) - \dim(Y) = 2.
\end{gather}

By Proposition~\ref{implies}
\begin{gather}\label{hx}
H^*(X;\R) = \R[u]/u^{i+1}, \quad \mbox{where} \ u = [\omega|_X].
\end{gather}
By \eqref{dims}, \eqref{ndims}, and  Lemma~\ref{euler},
\begin{equation}
 e^{S^1}(N_X) = 2 \left(t+\frac{u}{m}\right)^{i+1}, \quad \mbox{where} \ m = \phi(Y) - \phi(X).
\label{ex}
\end{equation}
Here, $N_X$ is the normal bundle to $X$ in $M$. Moreover, by
Lemma~\ref{r=1} and \eqref{ndims},
\begin{equation}\label{e2}
 e^{S^1} \big(N_X^{M^{\Z_2}}\big)= 2t.
\end{equation}
Since $ e^{S^1}(N_X) =  e^{S^1}\big(N_X^{M^{\Z_2}}\big) \, e^{S^1}\big(N_{M^{\Z_2}}\big)\big|_X,$ \eqref{ex}
and \eqref{e2} imply that
\begin{equation}\label{e1}
 e^{S^1}\big(N_{M^{\Z_2}}\big)\big|_X = \frac{1}{t} \left(t+\frac{u}{m}\right)^{i+1} =  t^i +
(i+1) \frac{u}{m} t^{i-1} + \dots + (i+1) \frac{u^i} {m^i} .
\end{equation}
By \eqref{e2}, \eqref{e1} and Lemma~\ref{Chern},
\begin{equation}\label{c12}
 c^{S^1} \big(N_{M^{\Z_2}})\big|_X = \frac{1}{1+t} \left(1+ t+ \frac{u}{m} \right)^{i+1}
\quad\mbox{and}\quad c^{S^1}\big(N_X^{M^{\Z_2}}\big) = 1 + 2 t.
\end{equation}
Therefore, since
$ c^{S^1}(N_X) =  c^{S^1}\big(N_{M^{\Z_2}})\big|_X \, c^{S^1}\big(N_X^{M^{\Z_2}}\big)$,
\begin{equation}\label{cnx}
 c^{S^1}(N_X) =  \left(1+t +
\frac{u}{m}\right )^{i+1} \frac{1+2t}{1+t}.
\end{equation}
By Lemma \ref{qchern}, \eqref{ndims} and \eqref{cnx} imply that,
\begin{equation}\label{cx}
\begin{split}
c(X)  &= \frac{
\left(1 + \frac{u}{m}\right)^i \left(1 + 2 \frac{u}{m}\right)
\left(1 - \frac{u}{m}\right)}
{1 - 2\frac{u}{m} } \\
&=
\frac{
\left(1 + \frac{u}{m}\right)^i }
{1 - 2\frac{u}{m} }  +
\frac{u}{m} \left( 1 + \frac{u}{m} \right)^i \\
& = 1 + (i+3)\frac{u}{m} + \dots + (3^i + i)\frac{u^i}{m^i}.
\end{split}
\end{equation}
By \eqref{hx} and  \eqref{cx}, the Euler characteristic of $X$ is
\begin{equation*}
i+1
= \sum_k (-1)^k\dim
(H^k(X;\R)) =
\int_X c_i(X) =
(3^i+i) \int_X \frac{u^i}{m^i}.
\end{equation*}
Therefore, $\frac{3^i + i}{i+1} \frac{u^i}{m^i} \in H^{2i}(X;\R)$
is a primitive integral class.
On the other hand, since $ e^{S^1}(N_X)$ is an integral class,
\eqref{e1} implies that
$(i+1) \frac{u^i}{m^i}$ is an integral class.
Combined, these two facts imply that
$\frac{(i+1)^2}{3^i+i}$ is an integer.
But this is impossible unless
\begin{equation}\label{i=1}
  i = 1,
\end{equation}
and so $2 \frac{u}{m} \in H^2(M;\R)$ is a primitive integral class.
By multiplying $\omega$ by a constant,
we may also assume that $[\omega]$ is a primitive integral class.
Hence, $u \in H^2(M;\R)$ is a primitive integral class by Lemma~\ref{primitive}.
Therefore,
\begin{equation}
\label{m2}
m = 2.
\end{equation}
Since nearly identical arguments can be applied to $Y$, the  claims
now follow from (\ref{dims}), (\ref{c12}), (\ref{cx}), (\ref{hx}),
(\ref{i=1}), and (\ref{m2}).
\end{proof}

\begin{lemma}\label{normal-bundle}
If the assumptions of Proposition~\ref{case2} hold and, additionally,
$\dim(M^{\Z_2})- \dim(Y) > 2$, then
\begin{align*}
\dim(X)  = \dim(Y) = 2i > 2  &\quad \mbox{and} \quad
\dim(M^{\Z_2}) = \dim(M) - 2 ;\\
H^*(X;\Z)/\torsion = \Z[u]/u^{i + 1} &\quad \mbox{and} \quad
H^*(Y;\Z)/\torsion = \Z[v]/v^{i+ 1}; \\
c(X) = (1+u)^{i+1} &\quad \mbox{and}  \quad
c(Y) = (1+v)^{i+1};  \\
c\big(N_{M^{\Z_2}}\big)\big|_X = 1+ u &\quad \mbox{and} \quad
c\big(N_{M^{\Z_2}}\big)\big|_Y = 1 + v; \\
c\big(N_X^{M^{\Z_2}}\big) = \frac {(1+u)^{i + 1}} {1 + 2u } & \ \
\mbox{and}   \ \
c\big(N_Y^{M^{\Z_2}}\big) = \frac {(1+v)^{i + 1}} {1 + 2v },
\end{align*}
where the last six equations are as elements of $H^*(X;\R)$ or of
$H^*(Y;\R)$.
Moreover, if $\phi(Y) > \phi(X)$ and $[\omega]$ is a primitive integral class, then
$\phi(Y) - \phi(X) = 2$.
\end{lemma}

\begin{proof}
Without loss of generality, we may assume that $\phi(X) < \phi(Y)$.
By Lemma~\ref{i<j}, since the action is not semifree, $\dim(X) =
\dim(Y)$; hence there
exists $i\in\N$ such that
\begin{gather}\label{bdims}
\dim(X) = \dim(Y) = 2i  \quad \mbox{and}\ \dim(M) = 4i + 2.
\end{gather}
Since $\dim\big(M^{\Z_2}\big)  - \dim(Y) > 2$,
Proposition~\ref{case3} implies that
\begin{equation} \label{bdims2}
\dim (M)-\dim \big(M^{\Z_2}\big)=2.
\end{equation}

By Proposition~\ref{implies},
\begin{gather}\label{bhx}
H^*(X;\R) = \R[u]/u^{i+1}, \quad\mbox{where}\,\, u = [\omega|_X].
\end{gather}
By \eqref{bdims} and \eqref{bdims2} and Lemma \ref{euler},
\begin{gather}
\label{bex}
 e^{S^1}(N_X) = 2^i\left(t + \frac{u}{m}\right)^{i+1},
\quad \mbox{where} \  m = \phi(Y) - \phi(X).
\end{gather}
Here, $N_X$ is the normal bundle to $X$ in $M$. Moreover, since
$\dim\big(M^{\Z_2}\big) - \dim(Y) > 2$, Lemma~\ref{1stChern} implies
that
\begin{equation}\label{be2}
 e^{S^1}\big( N_{M^{\Z_2}} \big)\big|_X =  t+ 2 \frac{u}{m}.
\end{equation}
Since $e^{S^1}(N_X) =
e^{S^1}\big( N_X^{M^{\Z_2}}\big) \,
e^{S^1}\big(N_{M^{\Z_2}}\big)\big|_X
$,
\eqref{bex} and \eqref{be2} imply
that
\begin{equation}\label{be1}
 e^{S^1}\big(N_X^{M^{\Z_2}}\big) = \frac{\left(2 t + 2 \frac{u}{m} \right)^{i+1}}{2t + 4
\frac{u}{m}}.
\end{equation}
By Lemma~\ref{Chern}, \eqref{be2} and \eqref{be1} imply  that
\begin{equation}\label{be2'}
 c^{S^1}\big(N_{M^{\Z_2}}\big) \big|_X= 1 + t+ 2 \frac{u}{m}
\quad \mbox{and} \quad
c^{S^1}\big(N_X^{M^{\Z_2}}\big)= \frac{(1+2t+2 \frac{u}{m})^{i+1}}{1+ 2t+ 4
\frac{u}{m}}.
\end{equation}
Therefore, since
$ c^{S^1}(N_X) =  c^{S^1}\big(N_X^{M^{\Z_2}}\big)  \,
c^{S^1}\big(N_{M^{\Z_2}}\big)\big|_X$,
\begin{equation}\label{bchern}
 c^{S^1}(N_X)
= \frac{(1+2t+2 \frac{u}{m})^{i+1}}{1+ 2t+ 4 \frac{u}{m}} (1 + t + 2
\frac{u}{m}).
\end{equation}
By  Lemma~\ref{qchern}, \eqref{bdims}, \eqref{bdims2}, and \eqref{bchern} imply that
\begin{equation}\label{C(X)}
c(X)=\left(1+2 \frac{u}{m} \right)^{i+1} = 1 + (i+1) 2 \frac{u}{m} + \dots +
(i+1) \left(2\frac{u}{m} \right) ^i.
\end{equation}
By \eqref{bhx} and \eqref{C(X)}  the Euler characteristic of $X$ is
$$ (i+1) 2^i \int_X \frac{u^i}{m^i} = \int_X c_i(X)=\sum_k
(-1)^k\dim H^k(X)=i+1.$$ So $ 2^i \frac{u^i}{m^i} \in H^{2i}(X, \Z)$ is a
primitive integral class.
By multiplying $\omega$ by a constant, we may assume that $[\omega]$
is a primitive integral class.
Hence,
$u$ is a
primitive integral class by Lemma~\ref{primitive}. On the one hand,
since $c^{S^1}\big(N_{M^{\Z_2}} \big) \big|_X$ is an integral class,
\eqref{be2'} implies that  $\frac{2}{m} \in \Z$. On
the other hand, since $u^i$ is an integral class and
$\frac{2^i}{m^i} u^i$ is a primitive integral class,
$\frac{m^i}{2^i} \in \Z$. Together, these imply that
\begin{equation}
\label{bm} m = 2 \quad \mbox{and}\quad u^i \,\,\mbox{is a primitive
integral class}.
\end{equation}
Since nearly identical arguments can be applied to $Y$,
claim now follows from
(\ref{bhx}),
(\ref{be2'}),
(\ref{C(X)}),
and
(\ref{bm}).
\end{proof}

\begin{lemma}\label{Q}
Let the circle act on a compact symplectic manifold $(Q, \omega)$
with moment map $\phi \colon Q \to \R$. Assume that there are
exactly two fixed components, $X$ and $Y$. Assume that
$\dim(X)=\dim(Y)=2i$ and $\dim(Q) = 4i$ for some $i > 1$. Let $\F$
be $\R$ or $\Z_p$, and when $\F=\Z_p$, we assume that the action is
semifree. Let
$$e^{S^1}(N_Y^Q) = \Lambda_Y \left( t^i  + \mu t^{i-1} + \lot \right)
\in H^{2i}_{S^1}(Y;\F),$$
where $N_Y^Q$ is the normal bundle to $Y$ in $Q$.
Assume also that there exist classes $\ut$ and $\mt$ in $H^2_{S^1}(Q;\F)$
such that
\begin{enumerate}
\item $\ut|_x = 0$ for all $x \in X$;
\item $\ut|_y \neq 0$ for all  $y \in Y$;
\item $\mt|_Y = \mu$; and
\item $\mt|_x \neq -it$ for all $x \in X$.
\end{enumerate}
Then $H^{2k+1}(X;\F) = 0$ for all $k$.
\end{lemma}

\begin{proof}

Assume on the contrary that there exists a non-zero class $\alpha \in
H^{2k+1}(X;\F)$ for some $k$.
By assumptions (1) and (2), there exist $u$ and $v$ in $H^2(X;\F)$
and a non-zero $m \in \F$ such that
\begin{equation}\label{utt}
\ut|_X = u \quad \mbox{and} \quad \ut|_Y = v + mt.
\end{equation}
Since $\F$ is a field, Poincar\'e duality implies that there exists
$\beta \in H^{2i - 2k- 1}(X;\F)$ such that $\alpha \cup \beta
 = u^i$.
Since $2k+1$ and $2i-2k-1$ are both smaller than $2\lambda_Y =
\dim(Q) - \dim(Y)$, there exist classes $\at \in
H^{2k+1}_{S^1}(Q;\F)$ and $\bt \in H^{2i - 2k - 1}_{S^1}(Q; \F)$
such that $\at|_X = \alpha$ and $\bt|_X = \beta$.

Since $H^{*}_{S^1}(Y;\F) =
H^*(Y; \F)[t]$,
we can write $\at|_Y =
\sum a_{2j+1} t^{k -j}$ and $\bt|_Y = \sum b_{2j+1} t^{i - k - j
-1}$, where $a_{2j+1}$ and $b_{2j+1}$ lie in $H^{2j+1}(Y;\F)$
for all $j$. Moreover, since $1 < 2i$,
there exist classes $\widetilde a_1\in
H^1_{S^1}(Q; \F)$ and $\widetilde b_1\in H^1_{S^1}(Q; \F)$ such that
$\widetilde a_1|_Y=a_1$ and $\widetilde b_1|_Y=b_1$. Finally, since
$H^1(\CP^\infty; \F)=0$  for all
$x\in X$, $\widetilde a_1|_x = \widetilde b_1|_x = 0$. Therefore,
\begin{gather}
\label{left1}
\big( \at \cup \bt \big)\big|_Y = \left( a_1 \cup b_1 \right)t^{i-1} +
\lot,  \\
\label{left2}
\big( \widetilde a_1 \cup  \widetilde b_1 \big) \big|_Y = a_1 \cup b_1,
\quad \mbox{and} \\
\label{left3}
\big(\widetilde a_1 \cup \widetilde b_1 \big)\big |_x = 0.
\end{gather}

Since the action is semifree when $\F=\Z_p$, we have  the short
exact sequence (\ref{short}) for both $\F=\R$ and $\F=\Z_p$. Since
$\dim(Q) - \dim(Y) = 2i$ and $ \ut^i\big|_X = \big( \at \cup \bt
\big)\big|_X$, by (\ref{short}), there exists $c \in \F$ so that
\begin{multline}
\label{right}
\big( \at \cup \bt \big)\big|_Y  =  \ut^i\big|_Y + c\, e^{S^1}(N_Y^Q) \\
= \left( m^i t^i  + i v m^{i-1} t^{i - 1}\right) +  \Lambda_Y \left(
c t^i + c \mu t^{i - 1} \right) + \lot .
\end{multline}
Comparing the highest order terms of \eqref{left1} and
\eqref{right}, we see that $$c\,
 \Lambda_Y = - m^i.$$
Hence, by comparing the next highest order terms, we see that
$$ a_1 \cup b_1 = m^{i-1} \left (iv - m \mu \right).$$
By \eqref{utt}, \eqref{left2}, and assumption (3), this implies that
$$\big( \widetilde a_1 \cup \widetilde b_1 \big) \big|_Y  =
m^{i-1} \big(i \ut - i mt - m \mt \big) \big|_Y. $$
Since $2 < \dim(Q) - \dim(X) = 2i$, this implies that
\begin{equation}\label{right2}
 \widetilde a_1 \cup \widetilde b_1   =
m^{i-1} \left(i \ut - i mt - m \mt \right) .
\end{equation}
But by \eqref{utt} and assumption  (4),
\begin{equation}
\label{right3}
m^{i-1} \big(i \ut - i mt - m \mt \big) \big|_x \neq 0 \quad
\forall \ x \in X.
\end{equation}
Clearly,
\eqref{left3}, \eqref{right2}, and \eqref{right3} give a contradiction.
\end{proof}

\begin{lemma}\label{torsion-Grass}
If the assumptions of
Proposition~\ref{case2} hold,
then $H^*(M^{S^1};\Z)$ is torsion-free.
\end{lemma}

\begin{proof}
By Lemma~\ref{isotropy-Q}, no
point in $M$ has stabilizer $\Z_k$ for any $k > 2$.
By Lemma~\ref{r=s=1}, the
cohomology $H^*(M^{S^1};\Z)$ is torsion-free if
$\dim \big(M^{\Z_2}\big) - \dim(Y)
= 2$,  and so we may assume that $\dim\big(M^{\Z_2}\big) - \dim(Y) > 2$. By
Lemma~\ref{normal-bundle},  $\dim(X)=\dim(Y)=2i$
and $\dim(M^{\Z_2}) = 4i$ for some $i
> 1$.

By Lemma \ref{normal-bundle}, $H^2(M;\R) = \R$.
Therefore, by multiplying $\omega$ by a constant,
we may assume that $[\omega]$ is a primitive integral class.
The induced effective $S^1 = S^1/\Z_2$ action on $M^{\Z_2}$ is semifree,
and the moment map for this action is $\phi'=\phi/2$.  By Lemma
\ref{normal-bundle}, $\phi(Y)-\phi(X)=2$, and so
$\phi'(Y)-\phi'(X)=1$. Hence, by Lemma ~\ref{uvt}, there exists an integral class $\ut
\in H_{S^1}^2(M^{\Z_2};\R)$, such that
\begin{equation}
\label{ut3} \ut\big|_X = [\omega|_X] \in H_{S^1}^2(X;\R) \quad \mbox{and}
\quad \ut\big|_Y = [\omega|_Y] - t \in H_{S^1}^2(Y;\R).
\end{equation}
In particular, for any prime $p$, there exists $\ut \in
H^2_{S^1}\big(M^{\Z_2}; \Z_p\big)$ such that $\ut|_x = 0$ for all $x
\in X$ and $\ut|_y \neq 0$ for all $y \in Y$.

By Lemmas~\ref{normal-bundle} and \ref{Chern}, the equivariant Euler
class of the normal bundle of $Y$ in $M^{\Z_2}$ (for the semifree
$S^1/\Z_2$ action on $M^{\Z_2}$) is given by
\begin{equation*}
\begin{split}
e^{S^1}\big({N_Y^{M^{\Z_2}}}\big) & =\frac{(-t+v)^{i+1}}{-t+2v} \\
& = (-1)^i \frac{(t-v)^{i+1}}{t}
 \left(1+\frac{2v}{t}+ \lot \right ) \\
&=(-1)^i\left (t^i + (1-i)v t^{i-1} +  \lot \right)
\ \in H_{S^1}^{2i}\big(M^{\Z_2};\R\big),
\end{split}
\end{equation*}
where $v = [\omega|_Y] \in H^2(Y;\R)$.  Moreover, by \eqref{ut3},
\begin{align}
(1 - i) \left(\ut +t \right) \big|_Y &= (1-i) v, \quad \mbox{and}\\
(1 - i) \left(\ut +t \right) \big|_x &= (1-i) t \quad \forall \ x \in X.
\end{align}

Finally,  fix any
prime $p$, and write
$$e^{S_1}(N_Y^{M^{\Z_2}}) = (-1)^i \left( t^i + \mu t^{i-1} + \lot \right) \in H^{2i}_{S^1}(Y;\Z_p),$$
where $\mu \in H^2(Y;\Z_p)$. Since $2 < \dim(M^{\Z_2}) - \dim(X) = 2i$,
there exists  a unique $\mt \in H^2_{S^1}(M^{\Z_2};\Z_p)$ such that $\mt|_Y
= \mu$. By the preceding paragraph, $\mt|_x = (1-i)t \neq -it \in
H^2(\CP^\infty;\Z_p)$. By Lemma~\ref{Q}, this implies that
 $H^{2k+1}(X; \Z_p)=0$ for all $k$ and all primes $p$.
By Lemma~\ref{torsion},
this proves the claim.
\end{proof}

\begin{remark}
In fact, we can use Lemma~\ref{betti0} to give
a simpler proof that
$H^{2k+1}(M^{S^1}; \Z_k)=0$ for all $k > 2$.
\end{remark}

\begin{appendix}
\section{Possible stabilizer subgroups}

The goal of this appendix is to prove the following proposition.

\begin{proposition}\label{general}
Let the circle act effectively on a  compact  symplectic manifold $(M,
\omega)$  with moment map $\phi \colon M \to \R$.
Assume that $M^{S^1}$ has exactly
two components.
Then no point has stabilizer $\Z_k$ for any $k > 6$.
\end{proposition}

\begin{proof}
Let $X$ and $Y$ be the fixed components. Let $\Xi_X$  denote the
multiset of weights for the isotropy action on the normal bundle to
$X$. By Corollary~\ref{same}, if the action is not semifree, the
multiset of weights for the isotropy action on the normal bundle to
$Y$ is $-\Xi_X$. By Lemma~\ref{modulo}, $\Xi_X=-\Xi_X\mod a$ for
each $a \in \Xi_X$. Finally, since the action is effective, the
weights in $\Xi_X$ are relatively prime. The result now follows
immediately from Lemma~\ref{II}.
\end{proof}

\begin{lemma}\label{II}
Let $W$ be a multiset of natural numbers which are relatively prime.
Assume that $W$ contains exactly $N$ distinct
numbers $a_1 < \dots < a_N$ which have
(non-zero)
multiplicities $m_1,\dots,m_N$, respectively.
Let $-W$ be the multiset of negative integers which
contains $-a_1,\dots,-a_N$ with the same multiplicity.
Assume that $W = -W \mod a_i$ for all $i \in \{1,\dots,N\}$.
Then
\begin{enumerate}
\item $a_i = i$ for all $i \in \{1,\dots,N\}$.
\item If $N = 3$, then $m_2 = m_1$.
\item If $N = 4$, then $m_3 = m_1$ and $m_2 = m_1 + m_4$.
\item If $N = 5$, then $m_4 = m_1 = 2 m_5$ and $m_3 = m_2 = 3 m_5$.
\item If $N = 6$, then $m_2 = m_3 = m_4 = 2 m_1 = 2 m_5 = 2 m_6$.
\item $N \leq 6$.
\end{enumerate}
\end{lemma}

\begin{proof}
The first claim follows immediately from Lemma~\ref{I}.
Now, the fact that $W = - W \mod N$ implies that
$m_i = m_{N-i}$ for all  $i \in \{1,\dots,N-1\}.$
Moreover, if $N > 3$,
the fact that $W = -W \mod (N-1)$ implies
$m_1 + m_N = m_{N-2}$ and $m_i = m_{N-i-1}$ for all $ i   \in \{2,\dots,N-3\}.$
Therefore,
\begin{gather}
\label{mod1}
m_1 = m_{N-1}, \quad \mbox{and} \\
\label{mod2}
m_2 = m_3 = \dots = m_{N-2} = m_1 + m_N \quad \forall \ N >  3.
\end{gather}

Claim (2) follows immediately from \eqref{mod1}, while claim (3)
follows immediately from \eqref{mod1} and \eqref{mod2}. If $N = 5$,
then since $W = -W \mod 3$,  $m_1 + m_4 = m_2 + m_5$.  Claim (4)
follows immediately from this fact and \eqref{mod1} and
\eqref{mod2}. Similarly, if  $N = 6$, then since   $W = - W \mod 4$,
$m_3 = m_1 + m_5$; claim (5) follows easily. Finally, if $N
> 6$, the fact that $W = -W \mod (N-2)$ implies that $m_2 + m_N =
m_{N-4}$, which contradicts \eqref{mod2}. This proves the last
claim.
\end{proof}

\end{appendix}

\end{document}